\documentclass[a4paper, 11.5pt]{amsart}
\usepackage{amssymb}
\usepackage[a4paper, total={6in, 8in}]{geometry}
\usepackage[utf8]{inputenc}
\usepackage[usenames, dvipsnames]{color}
\newtheorem{theorem}{Theorem}[section]
\newtheorem*{theorem*}{Theorem}

\newtheorem{lemma}[theorem]{Lemma}
\newtheorem{corollary}[theorem]{Corollary}

\newcommand{\C}{\mathbb C}
\newcommand{\R}{\mathbb R}

\newcommand{\D}{\mathbb D}
\newcommand{\T}{\mathbb T}
\newcommand{\B}{\mathbb B}

\newcommand{\Q}{\mathbb Q}

\usepackage{ulem}

\title[Radial limits of solutions to elliptic PDE's]{Radial limits of solutions to elliptic partial differential equations}

\dedicatory{}

\author[P. M. Gauthier]{Paul M. Gauthier}
\address[Paul M. Gauthier]{D\'epartement de math\'ematiques et de statistique, Universit\'e de Montr\'eal, Montr\'eal, Qu\'ebec, Canada H3C3J7.}
\email{paul.m.gauthier@umontreal.ca}

\author[M. Shirazi]{Mohammad Shirazi}
\address[Mohammad Shirazi]{Department of Mathematics and Statistics, Georgia State University, Atlanta, GA, USA.}
\email{mshirazi@gsu.edu}

\begin{document}

\dedicatory{}

\begin{abstract}
For certain elliptic differential operators $L,$  we study the behaviour of solutions to $Lu=0,$ as we tend to the boundary along radii in  strictly starlike domains in $\R^n, n\ge 3.$ Analogous results are obtained in other special domains. Our approach involves introducing harmonic line bundles as instances of Brelot harmonic spaces and  approximating continuous functions by harmonic functions on appropriate subsets. 
These approximation theorems on harmonic spaces yield interesting examples for approximation  by solutions of $Lu=0$ on some domains in $\R^n.$
\end{abstract}

\thanks{First author was supported by NSERC (Canada) grant RGPIN-2016-04107.}

\maketitle


\section{Introduction}
Denote by $\D$ the unit disc in the complex plane $\C$ and by  $\T$ the unit circle $\{z=e^{i\theta}:0\le\theta<2\pi\}.$ 
The motivation for this paper comes from the following classical result of Bagemihl and Seidel (see formula (8) in \cite{BS}). 

\begin{theorem}\label{Bagemihl-Seidel}
Let $E$ be  an $F_\sigma$ set of first category on the unit circle $\T.$ 
Then, for every function $\varphi$ continuous on the unit disc  $\D,$ there is a harmonic  function $u$ on $\D$ such that,  
$$ 
		\lim_{\rho \nearrow 1}\big(u-\varphi\big)\big(\rho e^{i\theta}\big)=0, \quad \mbox{for all} \quad e^{i\theta}\in E.
$$
\end{theorem}

We remark that, since every set of first category is contained in an $F_\sigma$ set of first category, we may drop the requirement that $F$ be an $F_\sigma$ set.

We  extend this result to solutions of certain partial differential equations. 
On a domain $\Omega\subset \R^n, \, n\ge 3,$ consider the following partial differential operator
\begin{equation}\label{L}
	L=\sum_{i,j=1}^na_{ij}\frac{\partial^2}{\partial x_i\partial x_j}+\sum_{i=1}^nb_i\frac{\partial}{\partial x_i}+c.	
\end{equation}

Following \cite{GGG1994_compact}, we assume that 
$a_{ij}=a_{ji}\in C^{2,1},$  $b_i\in C^{1,1},$ $c\in C^{0,1}$ and $c\le 0,$
where $C^{k,1}$ denotes the class of functions on $\Omega$ which are $k$ times continuously differentiable, the 
$k^{th}-$order
 partial derivatives being locally Lipschitz. 
We also assume that the associated quadratic form $(a_{ij})$ is positive definite on $\Omega,$ so $L$ is elliptic. A $C^2-$smooth solution
of $Lu=0,$ on an open subset $U\subset\Omega,$ will be called an $L-$harmonic function on $U$ and solutions (in the sense of distributions) of $Lu\le 0,$ which are lower semicontinuous, will be called $L-$superharmonic functions on $U.$ -
We assume that there is an $L-$superharmonic function $> 0$ on $\Omega$ which is not actually $L-$harmonic.

A subset $E\subset\Omega$ is said to be {$L-$\it thin at a point} $p\in \Omega$ if $p$ is not a limit point of $E$ or, when $p$ is a limit point of $E$, if there is an $L-$superharmonic function $u$ on an open neighbourhood of $p$, such that
$$
\liminf_{z\to p,\, z\in E} u(z)> u(p).
$$
A point $p$ in a subset $F$ of  $\Omega$ is said to be in the {\it $L-$fine interior} of $F,$ if $ \Omega\setminus F$ is $L$-thin at $p.$ 

 A subset $U\subset \Omega$ is said to be {\it $L-$finely open}
 in $\Omega,$ if all points of $U$ are $L-$fine interior points of $U.$ The family of $L$-finely open subsets of $\Omega$ determine a topology on $\Omega,$ called the {\it $L-$fine topology
} on $\Omega.$ 
The fine topology can be defined on a Riemannian manifold $M,$ where the Laplace-Beltrami operator $\Delta$ plays the role of the operator $L.$  In this case we shall often omit $\Delta$ from the notation and simply speak of fine thinness, fine openness, the fine topology, etc. In particular, there is a natural fine topology on submanifolds of $\R^n.$
If $\Omega=\R^{1},$
then the fine topology coincides with the standard topology on $\R^1.$

The following theorem is a natural higher dimensional analogue of Theorem \ref{Bagemihl-Seidel},  not for the ball, but for half-space $U=\R^{n-1}\times (0,+\infty)$ in $\R^n.$ This result is due to Essen and Gardiner \cite{EG} for the case that $L$ is the Laplacian $\Delta.$

\begin{theorem}\label{EG}
For $ n\ge 3,$   let $L$ be an operator of the form (\ref{L}) on  the upper half-space  $U=\R^{n-1}\times(0,+\infty)$ of $\R^n=\R^{n-1}\times\R,$  
let $E'=\cup_j E'_j$ be a countable union  of subsets of $\R^{n-1},$ such that the
Euclidean closure of each $E'_j$ has
empty fine interior in $\R^{n-1}.$
Then, for every function $\varphi$ continuous on $U,$ there is an $L-$harmonic  function $h$ on $U$ such that,
$$ 
\big(h-\varphi\big)(x^\prime,y)\to 0, \quad \mbox{if} 
\quad  y\searrow 0  \quad \mbox{for all} \quad x^\prime  =(x_1, \cdots, x_{n-1}) \in E'.
$$
\end{theorem}

Theorem \ref{EG} is a particular case of our Theorem \ref{vertical EG}, whose proof is based on an approximation theorem for solutions of partial differential equations, due to Gardiner, Goldstein and GowriSankaran \cite{GGG1994_compact}.

An analogue of this result can readily be formulated for radial limits in the ball. 
For $x\in \R^n\setminus\{0\},$ we have the polar representation $x=(\theta,\rho),$ where $\rho=|x|$ is in $(0,+\infty)$  and $\theta = x/|x|=x/\rho,$ with $\theta$ in the unit sphere $S^{n-1}.$ For $x=0,$ we may write $0=(\theta,0),$ with $\theta\in S^{n-1}$ arbitrary.  With these coordinates, $\R^n=S^{n-1}\times[0,+\infty)$ and the open unit ball has the representation $\B^{n}=S^{n-1}\times[0,1).$ 
In the present paper, as a special case of our main result, Theorem \ref{starlike EG},  we obtain the following result for $L-$harmonic functions.

\begin{theorem}\label{EG ball}
Let $L$ be an operator of the form (\ref{L}) on $\B^n,\, n\ge  3,$ satisfying the above conditions. 
Let  $E^\prime\subset S^{n-1}=\partial\B^n.$ Suppose  $E^\prime$ is expressed as $\cup_kE^\prime_k,$ where the Euclidean closure of each $E^\prime_k$ has empty fine interior in $S^{n-1}.$ 
Then, for each continuous funciton $g:\B^n\to\R,$ there is an $L-$harmonic function $u$ on $\B^n,$ such that 
$$
	\lim_{\rho \nearrow 1}\big(u-g\big)(\theta', \rho)=0,  \quad \mbox{for all} \quad \theta^\prime\in E^\prime.
$$

\end{theorem}

The proof of Theorem \ref{EG} in \cite{EG} for half-space and $L=\Delta,$ is based on a series of papers dealing only with half-space (and not the ball), so it is not simple to check whether the proof for half-space can be easily modified to give a proof for the ball. 
Our proof (even for classical harmonic functions) invokes a result of Fuglede in abstract potential theory on harmonic spaces.

To compare Theorem \ref{EG ball} with  Theorem \ref{Bagemihl-Seidel}, we have the following lemma.

\begin{lemma}\label{fine}
A subset of the unit circle $S^1$ is $S^1-$finely open if and only if it is open in the standard topology for $S^1.$ 
\end{lemma} 

\begin{proof}
It follows from the definition of the fine topology that every open set is finely open.  

Now, fix  a finely open set $U^\prime\subset S^1$ and put 
$U=U^\prime\times (0,+\infty) \subset\R^2\setminus \{0\}.$
 It is sufficient to show that every point $\theta$ in $U^\prime$ is in the $S^1-$interior of $U^\prime.$  Fix such a point $\theta,$ choose $\rho\in(0,+\infty)$  and put $z=(\theta,\rho)\in \R^2\setminus\{0\}.$ 
Since $U^\prime$ is an $S^1-$fine neighbourhood of $\theta,$ $S^1\setminus U^\prime$ is $S^1-$thin at $\theta.$ It follows from the result of Fuglede that,  $\big(\R^2\setminus\{0\}\big)\setminus U$ is thin at $z,$ and since thinness is local, $\R^2\setminus U$ is thin at $z.$ Since we are in $\R^2,$ there is a circle $C,$ centered at $z$ and of radius less than $|z|=\rho,$ with $C$ contained in the complement of $\R^2\setminus U,$ which is $U.$  Thus, the arc $\alpha$ of  $S^1,$ subtended by $C$ is contained in $U^\prime$ and is an $S^1-$neighbourhood of $\theta$ in $U^\prime.$  Hence, $\theta$ is an $S^1-$interior point of $U^\prime.$   

\end{proof}

From Lemma \ref{fine} it follows that, if we formulate the analogue of Theorem \ref{EG ball}, obtained by putting $n=2$ and $L=\Delta,$ we obtain the Bagemihl-Seidel Theorem \ref{Bagemihl-Seidel}. Hence Theorem \ref{EG ball} is indeed an extension of the Bagemihl-Seidel Theorem to higher dimensions.

As already mentioned, Theorem \ref{EG ball} is a special case of Theorem \ref{starlike EG}.
Recently, we found an analogue of Theorem \ref{Bagemihl-Seidel} for holomorphic functions of several complex variables \cite{GSFields} (see also \cite{KrMi}). 
The proofs in this paper, however, are independent of these results.
Some of the tools which we shall use were proved in the general context of harmonic spaces in axiomatic potential theory.


\section{Potential Theory}

Brelot \cite{Br1960} developped an axiomatic potential theory on certain topological spaces $\Omega.$  Suppose for a given  Hausdorff space  $\Omega,$ we associate to every open subset $U\subset\Omega$ a vector space $\mathcal H(U)$ of continuous functions on $U,$ which we call harmonic functions, satisfying the following Axioms $1, 2, 3$ of Brelot.   

{\it Axiom 1 (sheaf axiom).}  If $U$ is an open subset of $\Omega$ and  $u:U\to\R$ is harmonic on $U,$ then the restriction of $u$ to every open subset $V$ of $U$ is harmonic on $V$ and  if, for every point $x\in U,$ the restriction of $u$ to some open neighbourhood of $x$ is harmonic, then $u$ is harrmonic on $U.$ 

{\it Axiom 2.} $\Omega$ has a basis of open sets which are regular for the Dirichlet problem. 

{\it Axiom 3.} If $u_j$ is an increasing sequence of harmonic functions on a domain $U\subset\Omega$ and $u_j\to u,$ then either $u$ is harmonic on $U$ or $u\equiv+\infty.$ 

The pair $(\Omega, \mathcal H)$, where $\mathcal H$ denotes the  
sheaf $\mathcal H = \{\mathcal H(U): U\subset\Omega \, \mbox{ is open}\},$ is said to be a {\it Brelot (harmonic) space}.

On a harmonic space there is a natural notion of {\it superharmonic function} and 
by a {\it potential} on a harmonic space, we mean a superharmonic function $u \ge 0$ on $ \Omega$ for which every harmonic minorant $\ge 0$ vanishes identically \cite{Br1958}.

A {\it strong Brelot space} is a Brelot space in which there is a potential $>0.$ Note that $\R^2$ is {\it not} a strong Brelot space.  Indeed, if $u$ is a lower bounded superharmonic funciton on $\R^2,$ then $u$ is constant. In particular, every positive constant function $u>0,$ is a superharmonic function $u\ge 0,$ which is a harmonic minorant of itself not vanishing identically.  Thus, there are no potentials on $\R^2.$

With the following additional domination axiom, it is possible to adapt the greater part of classical potential theory to harmonic spaces. The original statement of the domination axiom would require many additional definitions, so we state an equivalent version for a strong Brelot space (see \cite[p. 228]{CC}.

{\it Axiom D (domination axiom).} For every locally bounded potential $p$ on $\Omega$ and every relatively compact open set $U$ in $\Omega,$ the (generalized) solution $H^U_p$ of the Dirichlet problem on $U$ with boundary values $p$ is the greatest harmonic minorant of $p|_U.$

{\it Axiom of polarity.}  A set $E\subset\Omega$ is polar provided that $E$ is thin at every point of $\Omega.$

Axion D implies the axiom of polarity (see \cite[Corollary 9.2.3]{CC}).

For the basic definitions and facts concerning potential theory, superharmonic functions, potentials,  thin sets, polar sets, the fine topology and harmonic approximation on harmonic spaces, see \cite{Br1960}, \cite{Hervethesis} and \cite{CC} and for the particular case of potential theory on Riemannian manifolds, we rely  on \cite{BB} and the references therein. 
One of the pleasant properties of superharmonic functions in harmonic spaces is the following {\it minimum principle} (see \cite[p. 428]{Hervethesis}).

\begin{lemma}\label{min} 
A superharmonic function $\ge 0$ on a domain is either $>0$ or 
$\equiv 0.$ 
\end{lemma} 

Notice that the constant function $0,$ is harmonic since $\mathcal H(\Omega)$ is a vector space. We can thus redefine a potential as a superharmonic function $u\ge 0$ on $\Omega,$ whose greatest harmonic minorant is $0.$ Since $0$ is harmonic, it is superharmonic and consequently  $0$ is a potential, since the greatest harmonic minorant of $0$ is itself. On Brelot spaces $(\Omega,\mathcal H)$ for which $0$ is not the only potential, the analogue of classical potential theory can be developped much further. Suppose $u$ is a potential other than $0$ on $\Omega$ and suppose $u(x)=0$ at some point $x\in\Omega.$   Then, by Lemma \ref{min}, $u=0,$ which contradicts the choice of $u.$  Thus, if there is a potential $u$ which is not the constant $0,$ then $u>0.$  
We have already named such a Brelot space a strong Brelot space.  In other words, a strong Brelot space is a Brelot space on which the constant $0$ is not the only potential.

A continuous map $\varphi : X\to Y$ between two Brelot spaces $X$ and $Y$ is called a {\it harmonic morphism} (see \cite{Fug}) if, for every open subset $V$ of $Y$ and every harmonic function $u$ on $V,$ the composition $u\circ\varphi$ is harmonic on $\varphi^{-1}(V)$ (when non-empty). A harmonic {\it isomorphism} between two Brelot  spaces is a homeomorphism which, along with its inverse, are harmonic morphisms.
The following lemma says that also the pullback of a superharmonic (respectively) subharmonic function by a harmonic morphism is superharmonic (respectively subharmonic). See \cite[Corollary 3.2]{CC1965}.

\begin{lemma}\label{superharmonic pullback}
If $\varphi:X\to Y$ is a nonconstant harmonic morphism, with $X$ connected, then  for every open subset $V$ of $Y$ and every superharmonic  (respectively subharmonic) function $u$ on $V,$ the composition $u\circ\varphi$ is superharmonic (respectively subharmonic) on $\varphi^{-1}(V)$ (when non-empty).
\end{lemma}

The following is proved in Fuglede's paper \cite{Fug}. 

\begin{theorem}\label{Fuglede Brelot}
Let $\varphi:X\to Y$ denote a surjective nonconstant harmonic morphism between connected Brelot spaces $X$ and $Y.$ For the ``if" parts of assertions (a), (b) and (c) below, suppose that the points of $Y$ are polar. For the ``only if" part of (b) and the ``if" part of (c), suppose in addition that $Y$ satisfies the axiom of polarity. 
\begin{enumerate}
\item[(a)] A set $V\subset Y$ is finely open if and only if the preimage $\varphi^{-1}(V)$ is finely open in $X.$ In other words, $\varphi:X\to Y$ is fine-to-fine open (being surjective)  and fine to-fine continuous.

\item[(b)] A set $E\subset Y$ is thin at a point $y\in Y$ if and only if $\varphi^{-1}(E)$ is thn at some (and hence every) point of the fiber $\varphi^{-1}(y)\subset X.$

\item[(c)] A set $E\subset Y$ is polar if and only if $\varphi^{-1}(E)$ is polar in $X.$
\end{enumerate}
\end{theorem}


\section{Harmonic approximation on Riemannian manifolds}\label{HARM}

We shall restrict the axiomatic potential theory to the case that $\Omega$ is a Riemannian manifold, by which we mean  a real, second countable, $C^\infty$-manifold with a Riemannian metric.  
We say that a subset of $M$ is {\it bounded} if it's closure is compact. 
For a subset $X\subset M,$ we say that a subset $H\subset M$ is a {\it hole} of $X$ if $H$ is a bounded component of $M\setminus X$ and we
define the topological hull $\widehat X$ of $X$ as the union of $X$ and the holes of $X.$ We shall say that $X$ is topologically convex if $X=\widehat X.$ 
By a regular exhaustion of $M$ by compact sets, we mean a sequence $(K_k)_k^\infty$ of topologically convex compact sets in $M,$ whose interiors cover $M$ and such that 
$$
	K_k \subset K_{k+1}^\circ, \quad k=1, 2, \ldots .
$$
If $M$ is not compact, then it has a regular exhaustion by compact sets.

By a harmonic function on an open subset $U$ of a Riemannian manfold $M,$ we mean a smooth solution $u$ on $U$ of the equation $\Delta u=0,$ where $\Delta$ is the Beltrami Laplacian on $M.$ Denote by $\mathcal H$ the sheaf of harmonic functions on $M.$ 
Then, $(M,\mathcal H)$ is an important instance of a Brelot space (see \cite[p. 112]{Fug}).

We recall the following simple lemma.

\begin{lemma}\label{liminf} 
If $u$ is superharmonic in an open set $U\subset M$, then for each $y$ in $U$,
$$
 \liminf_{x\to y} u(x) = u(y), \quad x \in U.
 $$
 \end{lemma}

It follows, that if $E$ is a subset of a Riemannian manifold $M$ and $u\in\overline E,$ then
$$
\liminf_{z\to y, z\in E} u(z)\ge u(y).	
$$
A subset $E$ is said to be {\it thin at a point} $y\in M$ if $y$ is not a limit point of $E$ or, when $y$ is a limit point of $E$, if there is a superharmonic function $u$ on an open neighbourhood of $y$, such that
\begin{equation}\label{thin}
\liminf_{z\to y, z\in E} u(z)> u(y).
\end{equation}
$E$ is said to be {\it nowhere thin} in a subset $G$ of $M,$ if $E$ is thin at no point of $G.$

The above lemma  follows from the semi-continuity and mean-value properties for superharmonic functions (\cite[p. 69]{BB}) and implies that every  open set in a Riemannian manifold of dimension $n>1$ is nowhere thin in itself. 
A set $E\subset M$ is said to be a {\it polar set} if there is a superharmonic
function $u$ on some open set 
$U$ such that $E \subset \{x \in U: u(x)=+\infty\}$.
An easy argument shows that if $E$ is polar and closed in $M,$ then the open set  $M\setminus E$ is nowhere thin in $M.$ A polar set in $\R^n$ is thin everywhere (\cite[Theorem 7.2.2.]{AG}).

As mentioned in the introduction, the fine topology can be defined on a Riemannian manifold using the Laplace-Beltrami operator $\Delta$ in place of the operator $L$ given by (\ref{L}) and we often simply speak of finely open, finely thin etc. instead of $\Delta-$finely open, $\Delta-$finely thin etc. The  fine topology on a Riemannian manifold $M$ is the smallest topology, for which all superharmonic functions are continuous as functions from $M$ to $(-\infty,+\infty].$ It follows from Lemma \ref{liminf} that the fine topology is finer than the initial topology on $M.$

From Lemma \ref{superharmonic pullback}, we have the following.

\begin{lemma}
Let $\varphi:M\to N$ be a non-constant harmonic morphism between Riemannian manifolds $M$ and $N.$ If a subset $E\subset N$ is thin at a point $y\in Y,$ then $\varphi^{-1}(E)$ is thin at each point $x\in \varphi^{-1}(y).$ 
\end{lemma}

\begin{proof}
Suppose $E$ is thin at a point $y\in Y$ and $x\in \pi^{-1}(y).$  If $y\not\in\overline E,$ it follows from the continuity of $\varphi$ that $x\not\in\overline{\pi^{-1}(E)},$
so the proof is done in this case. Now, if $y\in \overline E,$ there exists a function $u$ superharmonic in a neighbourhood $V$ of $y$ satisfying (\ref{thin}). Again, from the continuity, 
$$
\liminf_{w\to x,\, w\in \pi^{-1}(E)} (u\circ\varphi)(w)> (u\circ\varphi)(x).
$$
From Lemma \ref{superharmonic pullback}, $u\circ\varphi$ is superharmonic in the neighbourhood $\pi^{-1}(V)$ of $x,$ so $\pi^{-1}(E)$ is thin at $x.$

\end{proof}

In 1966, Cornea (see \cite[Corollary 5.1.1]{CC}) proved the following. 

\begin{lemma}\label{Baire}
$M$ endowed with the fine topology is a Baire space.
\end{lemma}

\begin{lemma}\label{A union B}
If $A$ and $B$ are closed sets with empty fine interior, then the same is true of $A\cup B.$ 
\end{lemma}

\begin{proof}
Let $V$ be a finely open subset of $A\cup B.$ We must show that $V$ is empty.  Since $V\cap (A\setminus B)=V\cap (M\setminus B),$ and $M\setminus B$ is finely open, being open, it follows that $V\cap(A\setminus B)$ is a finely open subset of $A$ and therefore empty.  Thus, the finely open subset $V$ of $A\cup B$ is contained in $B.$ Therefore it is contained in the largest finely open subset of $B,$ which is the fine interior of $B,$ which is empty.  That is, $V\cap B=\emptyset.$ We have 
$$
	V=V\cap(A\cup B)= \big(V\cap (A\setminus  B)\big) \cup (V\cap B) = \emptyset. 
$$
\end{proof}

We recall the following interesting fact.

\begin{lemma} \cite[Lemma A.1.6]{BW} \label{polar dimension}
If $M$ is a Riemannian manifold of dimension $m$ and $N$ is a submanifold of dimension $n$, then $N$ is polar in $M$ if and only $n\le m-2.$ 
\end{lemma}

In particular, a point in a Riemannian manifold of dimension $m$ is polar if and only if $m\ge 2$ and, since every subset of a polar set is polar, it follows that  there are no non-empty polar sets in $\R^1$ nor in $S^1.$

\begin{lemma}(see \cite[p. 112]{Fug})
For every Riemannian manifold $N,$ $(N,\mathcal H)$ is a Brelot space satisfying the axiom of polarity  and if the dimension of $N$ is greater than 1, points are polar.
\end{lemma}

\begin{proof}

In \cite[p. 112]{Fug}, Fuglede points out that  a Riemannian manifold is a Brelot space and satisfies an axiom, which he calls axiom $\overline D,$ and which is even stronger than the axiom of domination.

 We mentioned earlier that the axiom of domination implies the axion of polarity. If the dimension of $M$ is greater than 2, it follows from Lemma \ref{polar dimension} that points are polar.

\end{proof}

The following gives a simple example of a harmonic morphism (see \cite[Ch. 4]{BW}).

\begin{lemma}\label{product}
The {\it natural projection} $M\times N\to N$ from a Riemannian product is a harmonic morphism.
\end{lemma}

We can thus rephrase Fuglede's Theorem \ref{Fuglede Brelot} for Riemannian manifolds.

\begin{theorem}\label{Fuglede manifolds}
Let $\varphi:M\to N$ denote a surjective non-constant harmonic morphism between Riemannian manifolds $M$ and $N,$ where $N$ has dimension $>1.$ 
\begin{itemize}
\item [(a)]  A set $V\subset N$ is finely open if and only if the preimage $\varphi^{-1}(V)$ is finely open in $M.$ In other words, $\varphi:M\to N$ is fine-to-fine open (being surjective) and fine-to-fine continuous.

\item [(b)] A set $E\subset N$ is thin at a point $y\in N$ if and only if $\varphi^{-1}(E)$ is thin at some (and hence every) point of the fiber $\varphi^{-1}(y)\subset M.$

\item [(c)] A set $E\subset N$ is polar if and only if $\varphi^{-1}(E)$ is polar in $M.$
\end{itemize}
\end{theorem}

\begin{corollary}\label{fine interior}
Under the same hypotheses, a set $E\subset N$ has empty fine interior if and only if $\pi^{-1}(E)$ has empty fine interior. 
\end{corollary}

\begin{proof}
$E\subset N$ has non-empty fine interior if and only if $N\setminus E$ is thin at some poiint $y\in E$ if and only if $\pi^{-1}(N\setminus E)=M\setminus\pi^{-1}(E)$ is thin at some point $x\in \pi^{-1}(y)\subset
\pi^{-1}(E)$ if and only if $\pi^{-1}(E)$ has non-empty interior. 

\end{proof}

Euclidian space $\R^n$  is a Riemannian manifold with the usual distance $d_n$ as Riemannian metric.  We consider the sphere $S^{n-1}$ as a Riemannian manifold with the restriction of $d_n$ as Riemannian metric.  This induces a harmonic structure on $S^{n-1}$ and we shall sometimes add the prefix ``spherically" to the corresponding potential theoretic notions on $S^{n-1},$ for example spherically thin or spherically polar sets.

\begin{lemma}\label{fine sphere}
For $n>2,$ in the fine topology of $\R^n,$ the fine interior of $S^{n-1}$ is empty, but the spherically fine interior of $S^{n-1}$ is all of $S^{n-1}.$
\end{lemma}

\begin{proof}
The second part is trivial, because $S^{n-1}$ is spherically open, hence spherically finely open. 
For the first part, because of the symmetry of the sphere,  $\R^{n-1}\setminus S^{n-1}$ is thin either at every point or no point of $S^{n-1}.$ Suppose, to obtain a contradiction, that $\R^{n-1}\setminus S^{n-1}$ is thin at every point of $S^{n-1}.$ Then, every point of $S^{n-1}$ is in the fine interior of $S^{n-1},$ so $S^{n-1}$ is a finely open subset of $\R^{n}.$ Since $S^{n-1}$ is closed in $\R^n,$ it is finely closed in $\R^n.$ Thus, $S^{n-1}$ is a finely  {\it clopen} (closed and open) subset $\R^n.$ But $\R^n$ is connected in the fine topology, so the only finely clopen subsets are $\emptyset$ and $\R^n.$ This is the desired contradiction and so $\R^n\setminus S^{n-1}$ is thin at no point of $S^{n-1}.$ That is, no point of $S^{n-1}$ is in the fine interior of $S^{n-1}.$

\end{proof}

If $\Omega$ is a locally compact Hausdorff space, we denote by $*_\Omega,$ or simply $*$ if the context is clear, the ideal point of the one-point compactification $\Omega^*=\Omega\cup \{*\}$ of $\Omega.$ We say that a subset $X\subset \Omega^*$ is {\it bounded} if the ideal point $*_\Omega$ is not in the closure of $X.$ If $X\subset \Omega,$ this is consistent with our previous definition of boundedness.
For $F$  a closed subset of a Riemannian manifold $M.$ We define (as earlier for harmonic spaces) a {\it hole} of $F$ to be  any bounded component of $M\setminus F.$  
We denote by $\widehat F$ the union of $F$ and all of its holes. Thus, $\widehat F$ is a closed set having no holes.

We denote by $A(F)$ the family of continuous functions on $F$ which are harmonic on the interiour $F^\circ.$ 
A closed set $F\subset  M$ is said to be a 
{\it Mergelyan} (respectively {\it Carleman}) set in $M$ if for every $u\in A(F)$,  and every $\epsilon,$ which is a  positive constant (respectively, positive continuous function on $F$)  there exists a function $h$, harmonic on $M$, such that $|h-u|<\epsilon$ on $F.$
The terminology is in honour of Mergelyan, who showed (see \cite{Gai}) that, if $K$ is a compact subset of $\C$ with connected complement 
(i.e. $\widehat K=K$), then, for each function $f$ continuous on $K$ and holomorphic on $K^\circ$ and each positive constant $\epsilon,$ there is a polynomial $p$ such that $$|f(z)-p(z)|<\epsilon, \quad \textit{for all} \quad z\in K$$ and Carleman, who showed that, for each continuous function $f(x), \, x\in \R$ and for each positive continuous function $\epsilon(x), \, x\in\R,$ there is an entire function $g(z), z\in\C,$ such that 
$$|f(x)-g(x)|<\epsilon(x), \,\quad \textit{ for all} \quad x\in\R.$$
 Note that by the Tietze Extension Theorem, the function $\epsilon,$ in the definition of a Carleman set, may be considered to be defined on all of $M,$ not just on $F.$

The following theorem from \cite{BG} gives a characterization of Mergelyan sets and 
also of Carleman sets having no interior.

\begin{theorem}\label{BG} 
Let $F$ be a closed subset of a non-compact Riemannian manifold $M$. Then the
following are equivalent.
\begin{itemize}
\item[(a)]  $F$ is a Mergelyan set.

\item[(b)]  All three of the following conditions hold:

(i) $M\setminus \widehat F$ and $M\setminus F^\circ$ are thin at the same points of $F$.

(ii) $M^\star\setminus \widehat F$ is locally connected.

(iii) The holes of $F$ satisfy the long islands condition.
\end{itemize}
Moreover, if $F$ is a Mergelyan set and $F^\circ=\emptyset,$ then $F$ is a Carleman set. 
\end{theorem}


\section{Approximation on harmonic bundles}

Let $(M,g)$ be a Riemannian manifold with Riemannian metric $g.$ We denote by $(M,\mathcal H)$ the associated Brelot harmonic space, where it is understood, by abuse of notation, that $\mathcal H$ is the sheaf $\mathcal H_g$ of harmonic functions on $M.$ 
A continuous map $\phi : M\to N$ between two Riemannian manifolds 
$(M,g)$ and $(N,h)$ is called a {\it harmonic morphism} (see \cite[Def. 4.1.1]{BW}) if it is a harmonic morphism between the associated Brelot spaces $(M,\mathcal H_g)$ and $(N,\mathcal H_h).$ 

A simple example of a harmonic isomorphism is the change of charts on a Riemannian manifold of dimension $> 1.$ 

The following quote is from \cite[Ch. 12]{BW}. ``The general problem of classifying harmonic morphisms between [Riemannian] manifolds of arbitrary dimensions remains far from our reach at the present time; however, the problem becomes tractable when they map from an $n$-dimensional to an $(n-1)$-dimensional manifold, so that the regular fibres are one-dimensional."

Fuglede, in the next theorem, furnishes two important instances, where the fibers are one-dimensional.

\begin{theorem}\cite[Theorems  2 and 3]{Fug}\label{Fug} 
For $n>1,$ let
$\varphi:M\to  N$ 
denote either the ``vertical projection"  $\R^n=\R^{n-1}\times\R\to \R^{n-1}$ or the ``radial projection" $\R^n\setminus\{0\}\to S^{n-1},$ given by $y\to y/|y|.$ Then, $\varphi$ is a harmonic morphism  and a set $E\subset N$ is thin at a point $x\in N$ if and only if $\varphi^{-1}(E)$ is thin at some (and hence every) point of the fiber $\varphi^{-1}(x)\subset M.$ Moreover, a set $E\subset N$ is polar iff $\varphi^{-1}(E)$ is polar.
\end{theorem}

We remark here that the case $\R^n\to \R^{n-1}$ was first proved by Armitage and Gardiner \cite[Theorem 7.8.6]{AG}.

\bigskip

The above two projections suggest that we introduce 
{\it harmonic line bundles}. By a line bundle, we mean a  fiber bundle $\big(M,\, N, \, \pi, \, I\big),$ where $M$ and $N$ are Riemannian manifolds and the fiber $I$ is an open interval $a<t<b,$ where $\infty\le a<b\le+\infty.$ Thus, for every point $x\in N,$ there is an open neighbourhood $V_x\subset N$ and a homeomorphism $\psi_x : \pi^{-1}(V_x)\to V_x\times I$ called a {\it local trivialization}, such that $\pi_x\circ \psi_x=\pi,$ where $\pi_x$ is the usual projection from $V_x\times I$ onto $V_x.$ If $a\in V_x\times V_y,$ we require that 
$$\psi_y\circ\psi^{-1}_x:\{a\}\times I\to\{a\}\times I$$
 be an order isomorphism.  
Hence the fiber $I_x=\pi^{-1}(\pi(x))$ containing $x$ can be considered as an {\it ordered} open arc properly embedded in $M.$ 
Every (non-empty and open) connected subset $C$ of $I_x$ is an ordered subarc  
$(\alpha_C,\beta_C)\subset I_x,$ with $\alpha_C<\beta_C.$
If $\beta=+\infty,$ for 
$$\psi_x\big((\alpha_C,\beta_C)\big)=\{x\}\times(\alpha,\beta)\subset\{x\}\times I,$$
then  the same is true for every $\psi_y,$ where $C\subset \pi^{-1}(V_y).$  In this case we write $\beta_{C}=+\infty,$ without ambiguity.  The same for $\alpha_{C}=-\infty.$

Since $I_x$ is closed in $M,$ the $I$-closure of a subset $A$ of $I_x$ is the same as the $M$-closure of $A$ and so there is no ambiguity in the notation $\overline A.$  
In particular, $A$ is bounded in $I_x$ if and only if it is bounded in $M,$ so there is no ambiguity in the notion of a bounded subset of $I_x.$ It follows from the preceding that a bounded connected subset $H$ of $I_x$ is an arc $(\alpha_x,\beta_x),$ with $\alpha_x,\beta_x\in I_x.$ That is, $\alpha_x\not=-\infty$ and $\beta_x\not=+\infty.$ 

\medskip

Let $\big(M,\, N, \, \pi, \, I\big)$ be a line bundle.
By Lemma \ref{product} the projections $\pi_x$ are harmonic morphisms.  By a {\it harmonic line bundle}, we mean such a line bundle, where the local homeomorphisms $\psi_x$ are harmonic isomorphisms. If we have a harmonic line bundle, it follows that the projection $\pi:M\to N$ is a harmonic morphism, being locally a combination of harmonic morphisms. 
We warn the reader that the expression ``harmonic line bundle" may have different meanings in the literature.  
For the remainder of this section $\big(M,\, N, \, \pi, \, I\big)$ is a harmonic line bundle, which we shall by abuse of notation often denote simply by $M.$

\bigskip

For $X\subset M$ and $x\in X,$ denote by $X_x$ the {\it slice}
 $X_x=X\cap I_x.$ 
We define the {\it fiber hull} of $X,$ denoted by $X^I,$ as the union of $X$ and all of the holes of 
all of the slices $X_x$ over all $x\in X.$ More precisely, 
$$
	X^I = X\cup\bigcup_{x\in X} \left\{H: H\, \mbox{is a bounded component of } \, I_x\setminus X_x=I_x\setminus X\right\}.
$$

It is easy to see that, if $X$ is closed then $\widehat X\subseteq X^I.$ That is, the topological hull is contained in the fiber-hull.

\begin{lemma}\label{monotone}
If $X\subset Y\subset M,$ then $X^I\subset Y^I.$ 
\end{lemma}

\begin{proof}
Fix $x\in X^I.$ If $x\in Y,$ then obviously $x\in Y^I.$ If $x\not\in Y,$ let $H_y$ be the component of $I_x\setminus Y$ containing $x.$ Then $H_y$ is a connected subset of $I_x\setminus X$ containing $x$ and so $H_y$ is contained in the largest connected subset $H_x$ of $I_x\setminus X$ containing $x.$ That is, $H_y\subset H_x.$ Since $H_x$ is bounded, it follows that $H_y$ is bounded.  Thus, $x \in H_y\subset Y^I.$ 

\end{proof}

Denote by $\alpha_{x}$ and $\beta_{x}$ respectively the infimum and supremum of points in $I_x\cap X,$ using the induced order. 
It follows that, if $X$ is compact, then for each $x,$ $I_x\cap X^I$ is the (possibly degenerate) arc $[\alpha_x,\beta_x]$ in $I_x.$ The fiber hull $X^I$ contains all (possibly degenerate) bounded arcs in $I_x$ whose end points lie in $X$ and it is not hard to see that, if $X$ is closed, then $X^I$ is precisely the union of all  (possibly degenerate) compact arcs in $I_x$ whose end points lie in $X.$

\begin{lemma}\label{compactopenhull}
\begin{equation}\label{compacthull}
	X^I=\bigcup_{x\in X}
	[\alpha_x,\beta_x],
	 \quad \mbox{for $X$ compact},
\end{equation}
and
\begin{equation}\label{openhull}
	X^I=\bigcup_{x\in X}
	(\alpha_x,\beta_x),
	 \quad \mbox{for $X$ open.}
\end{equation}
\end{lemma}

\begin{proof}The proof of  (\ref{compacthull}) is straigtforward. To prove (\ref{openhull})
let $y\in \cup_{x\in X} (\alpha_x, \beta_x).$ Then, for some $x\in X$ we have $y\in (\alpha_x, \beta_x).$ If $y\in X,$ then clearly $y\in X^I.$ If $y\not\in X,$ choose $\alpha,\beta\in X,$ with $\alpha_x<\alpha<y<\beta<\beta_x.$ The component of $I_x\setminus X$ containing $y$ is bounded, since it is contained in $(\alpha,\beta).$ 
Thus, $y\in X^I$ by definition.

Now let $y\in X^I.$ Then $y\in X$ or it belongs to a bounded component $[\alpha,\beta]$ of $I_z\setminus X$ for some $z\in X.$ In the first case, 
since $y$ is in the interior of $X,$ there exist $\alpha,\beta\in X\cap I_y,$ with $\alpha_y<\alpha<y<\beta<\beta_y$ and $[\alpha,\beta]\subset X$ so
$y\in (\alpha_y, \beta_y).$ Thus $y\in \cup_{x\in X} (\alpha_x, \beta_x).$ In the second case,  there exist $\gamma, \delta\in X\cap I_z, \, \alpha_z<\gamma<\alpha\le\beta<\delta<\beta_z.$ Clearly $y\in (\alpha_z, \beta_z),$ 
so it is in $\cup_{x\in X} (\alpha_x, \beta_x).$ 

\end{proof}

\medskip

\begin{lemma}\label{fiberhull}
If $X\subset M$  is compact,  $X^I$ is compact.
 \end{lemma}
\begin{proof}
 Suppose $X$ is compact and let $\{x(\mu)\}_\mu$ be a net in $X^I$ converging to a point $x(0)\in M.$ If $\{x(\mu)\}_\mu$ has a subnet in $X,$ then $x(0)\in X,$ since $X$ is closed and it follows that  $x(0)\in X^I,$ since $X\subset X^I.$ 
If each $x(\mu)\not\in X,$ let $H_\mu=(\alpha_\mu,\beta_\mu)$ be the bounded component of $I_{x(\mu)}\setminus X$ containing $x(\mu),$ where $\alpha_\mu, \beta_\mu\in I_{x(\mu)}.$  
If either $\alpha_\mu$ or $\beta_\mu$ were not in $X,$ then $H_\mu$ would not be maximally connected. Thus, $\alpha_\mu, \beta_\mu\in X.$
Since $X$ is compact, we may assume that $\alpha_\mu\to \alpha$ and $\beta_\mu\to \beta,$ with $\alpha,\beta \in X.$ 
For $y_\mu=\pi(\alpha_\mu)=\pi(\beta_\mu),$ we have that $y_\mu\to y=\pi(\alpha)=\pi(\beta),$ with 
$x(0)$ in the subarc of $\pi^{-1}(y)$ with end points $\alpha$ and $\beta.$ Since the arc between two points of $X,$ lying on the same fiber is included in $X^I,$ it follows that $x(0)\in X^I,$
which completes the proof.

\end{proof}

Let us say that $X$ is fiber-convex if it is equal to its fiber hull, that is $X=X^I.$  From Lemma \ref{fiberhull} and Equation (\ref{compacthull}), we have the following.

\begin{lemma}\label{fiberhullcompact}
If $X\subset M$ is compact, then $(X^I)^I=X^I.$ That is, the fiber hull of a compact set is fiber convex.  
\end{lemma}

\medskip

\begin{theorem}\label{regularfiberconvex}
If $M$ is not compact, then it has a regular exhaustion by fiber-convex compact sets.
\end{theorem}
\begin{proof}
Choose a regular exhaustion $(K_k)_{k=1}^\infty$ of $M$ by compact sets
and consider the family $(K_k^I)_{k=1}^\infty.$
We claim that
 
(1) Each $K_k^I$ is compact.

(2) $\bigcup (K_k^I)^\circ=M.$

(3) $K_k^I\subset (K_{k+1}^I)^\circ.$

(4) Each $K_k^I$ is fiber-convex; that is, ${(K_k^I)^I}=K_k^I.$

\smallskip
 
The first one follows from Lemma \ref{fiberhull}. The second is obvious. The last one follows from Lemma \ref{fiberhullcompact}.
Now, to prove the third,  by  Lemma \ref{compactopenhull}, and the fact that $K_k\subset K_{k+1}^\circ, k=1, 2, \cdots,$ we have 
$$
	K^I_k = \bigcup_{x\in K_k}[\alpha_{x}^k,\beta_{x}^k] \subset
		\bigcup_{x\in K_{k+1}^\circ} (\alpha_x^{k+1},\beta^{k+1}_x) = (K_{k+1}^\circ)^I.
$$

\end{proof}

\medskip

\begin{theorem}\label{bundle} 
Let $(M,\, N, \, \pi, I)$ be a harmonic line bundle over $N,$ with $M$ non-compact. Let $F^\prime =\cup_kF^\prime_k,$ where each $F^\prime_k$ is a subset of $N,$ whose closure
has empty fine interior. Then, for every continuous real-valued function $\varphi$ on $M,$ there is a harmonic function $h$ on $M,$ such that, for all 
$x'\in F^\prime,$ 
$$
	\big(h-\varphi\big)(m)\to 0, \quad \mbox{as} \quad  m\to *_M, \quad \mbox{with} 
		\quad m\in\pi^{-1}(x'). 
$$
\end{theorem}

\begin{proof}
We may assume that each $F_k'$ is a  closed set in $N,$ with empty fine interior and by Lemma \ref{A union B}, we may  also assume that the $F^\prime_k$ are increasing, $F_k^\prime\subset F_{k+1}^\prime.$
By Theorem \ref{regularfiberconvex}, there is a regular exhaustion $K_k, \, k=1,2,\ldots,$ of $M$ by fiber-convex compact sets. For each $k,$ set 
$$
	F_k= \pi^{-1}(F_k^\prime)\cap(M\setminus K^\circ_{k})  \quad \mbox{and} 
		\quad F=\bigcup_{k=1}^\infty F_k.
$$
We now verify that the conditions in part (b) of 
Theorem \ref{BG} are satisfied, so $F$ is a Mergelyan set in $M.$

The set $F$ is closed, since $F$ is the union of a locally finite family $(F_k)_{k=1}^\infty$ of closed sets. 
 Since each $F_k^\prime$ has empty fine interior, it follows from Lemma \ref{Baire} that $F^\prime$ also has empty fine interior.  From Corollary \ref{fine interior}, we have that $\pi^{-1}(F^\prime)$ has empty fine interior and hence the smaller set $F$ also has empty fine interior and ({\it a fortiori}) empty interior.

\bigskip

{\bf Part $(i).$} We must check that $M\setminus \widehat F$ and $M\setminus F^\circ$ are thin at the same points of $F.$

First of all, we  claim $\widehat F=F$ ($F$ has no holes). 
Fix $x\in M\setminus F$ and let $C_x$ be the component of $M\setminus F$ containing $x.$ It is sufficient to show that $C_x$ connects $x$ to $*.$  Since the projection $\pi$ is open, $\pi(C_x)$ is an open subset of $N$ and hence contains a point $y\in N\setminus F^\prime.$  The fiber $\pi^{-1}(y)$ is a (topological) line which meets $C_x$ and which is unbounded.  Thus, it connects $C_x$ to $*.$ Thus $M^*\setminus F$ is connected. That is, $\widehat F=F.$

Since  $\widehat F=F,$ we have that  $M\setminus \widehat F=M\setminus F$ and since $F^\circ=\emptyset,$ we have that $M\setminus F^\circ = M.$ Thus, we must show that $M\setminus F$ and $M$ are thin at the same points of $F.$ 
First note that every point of $F$ is an interior point of $M,$ so $M$ is thin at no point of $F,$ by the remark following Lemma \ref{liminf}.

Since $\pi^{-1}(N\setminus F')$ is a subset of $M\setminus F,$ it is sufficient to show that $\pi^{-1}(N\setminus F')$ is non-thin at every point of $F.$ 
Let $y\in M,$ be an arbitrary point in $F.$ By our definition, $y\in \pi^{-1}(F').$ Thus  $y\in \pi^{-1}(x'),$ for some $x'\in F'.$ Now by Theorem \ref{Fuglede manifolds}, $\pi^{-1}(N\setminus F')$ is non-thin at $y$ if and only if $N\setminus F'$ is non-thin at $x'.$ 
The latter follows from our assumption on $F'.$  
Thus, the condition $(i)$ is fulfilled.

\bigskip

{\bf Part $(ii).$} 
Since $M\setminus F$ is an open subset of the manifold $M,$ it is locally connected (at each point of $M\setminus F$) and since $M\setminus F$ is open in $M^*\setminus F,$ it follows that $M^*\setminus F$ is locally connected at each point of $M\setminus F.$  There remains  only to check that $M^*\setminus F$ is locally connected at the ideal point $*.$ Since the sets  $(M^*\setminus F)\setminus K_k, \, k=1,2, \ldots,$ form a neighbourhood basis of the point $*$ in $M^*\setminus F,$ it is sufficient  to show that each 
$(M^*\setminus F)\setminus K_k$ is connected. Now, 
\begin{equation}
	\big(M^*\setminus \pi^{-1}(F')\big)\setminus K_k \subset (M^*\setminus F)\setminus K_k.
\end{equation}
Since $\pi^{-1}(F^\prime)$ is nowhere dense, it follows that the left side is dense in the right side, so it is sufficient to show that the left side is connected, because every set between a connected set and its closure is also connected.  
For the left side, we have
\begin{equation}\label{hull}
	\left(M^*\setminus \pi^{-1}(F')\right)\setminus K_k = 
	\bigcup_{x'\not\in F'}\big[\{*\}\cup\big(\pi^{-1}(x')\setminus K_k\big)\big]  . 
\end{equation}

For all $x^\prime\in N, \ \pi^{-1}(x')$ is an unbounded topological line. If $\pi^{-1}(x^\prime)\cap K_k=\emptyset,$ then certainly $\pi^{-1}(x')\setminus K_k=\pi^{-1}(x')$ is an unbounded connected set. If $\pi^{-1}(x^\prime)\cap K_k\not=\emptyset,$ then $\pi^{-1}(x')\setminus K_k$ is
a union of two unbounded topological lines.
This is because of our choice of compact sets (each one is equal to its fiber hull). 
Thus, each $\pi^{-1}(x')\setminus K_k$ is either an unbounded connected set or the union of two unbounded connected sets.
Since every set between a connected set and its closure is connected, the union with $*$ of each of these unbounded connected sets is also connected. 
Thus the right member of (\ref{hull}) is connected, since it is a union of connected sets with a common point $*.$ 
This concludes the proof of Part $(ii).$

\medskip

{\bf Part $(iii).$} Note that $F$ has no holes, so there is nothing to verify.

\medskip

We have shown that $F$ is a Mergelyan set, and since $F^\circ =\emptyset,$ it is in fact a Carleman set in $M$, so for $\varphi$  continuous on $M,$
the existence of the harmonic function $h$ 
on $M,$ 
satisfying the approximation property
$$|h(m)-\varphi(m)|\rightarrow 0, \quad \textit{ as } m\to *_M,\quad m\in \pi^{-1}(x'),$$
for all $x'\in F',$ follows from the definition of a Carleman set (in $M$). 

\end{proof}

The following is the most important instance of Theorem \ref{bundle} and was first proved by Essen and Gardiner \cite[Theorem 3]{EG}.

\begin{theorem}
For $n>2,$
let $F^\prime =\cup_kF^\prime_k,$ where each $F^\prime_k$ is a subset of $\R^{n-1},$ whose closure has empty fine interior. Then, for every continuous real-valued function $\varphi$ on the upper half-space $\R^{n-1}\times (0,+\infty),$ there is a harmonic function $h$ on $\R^{n-1}\times (0,+\infty),$ such that, for all $x^\prime\in F^\prime,$ 
$$
	\big(h-\varphi\big)(x^\prime,t)\to 0, \quad \mbox{as} \quad   t\to 0.
$$
\end{theorem}

We also have an analogous result, replacing the half-space by the ball and vertical projection by radial projection.

\begin{theorem}
For $n>2$ and $0<R\le+\infty,$ let $\B^n_R$ be the open ball of center $0$ and radius $R$ in $\R^n.$ 
Let $F^\prime =\cup_kF^\prime_k,$ where each $F^\prime_k$ is a subset of $S^{n-1},$ whose closure
has empty fine interior. Then, for every continuous real-valued function $\varphi$ on $\B^n_R,$ there is a harmonic function $h$ on $\B^n_R,$ such that, for all 
$\theta'\in F^\prime,$ 
$$
	\big(h-\varphi\big)(r, \theta')\to 0, \quad \mbox{as} 
		\quad r\nearrow R.
$$
\end{theorem}
As pointed out in the introduction, it is not easy to see how to modify the proof of Essen and Gardiner  in  \cite[Theorem 3]{EG} to obtain a proof for the ball. Moreover, 
they used the Possion Kernel which we shall not use in our proof.

\begin{proof} 
Construct a sequence $(r_k)_{k=1}^\infty,\ 0<r_k<r_{k+1},\ r_k\to R$ and let 
$$F_k=\{r\theta'\ :\ r_k\le r<R,\ \theta'\in F_k^\prime\}$$
 and $F=\cup_kF_k.$ We claim that, by Theorem \ref{BG}, $F$ is a Carleman set in $\B_R^n.$ The proof of (ii) is essentially the same as in the proof of Theorem \ref{bundle} and (iii) is trivial, since $F$ has no holes.

To prove (i), let $\pi : \R^n\setminus\{0\}\to S^{n-1}$ be the natural projection. Since $F$ has no holes in $\B^n_R\setminus\{0\},$  $(F)^\circ=\emptyset,$ and locally, $F=F_k,$ for some $k,$ it follows from Theorem \ref{Fug} that $(\B^n_R\setminus\{0\})\setminus F$ and $\B^n_R\setminus\{0\}$ are thin at the same points of $F.$ Since $\B^n_R\setminus\{0\}$ is open in $\B^n_R,$  $F$ is closed in $\B^n_R,$ $0\not\in F,$ $F=F_k$ locally in $\B^n_R,$
 for some $k$ and thinness is a local property, it follows that $\B^n_R\setminus F$ and $\B^n_R$ are thin at the same points of $F.$ Since $F=\widehat F$ in $\B^n_R$ and $F^\circ=\emptyset,$ this proves (i) and concludes the proof that $F$ is a Mergelyan set in $\B^n_R$ and in fact a Carleman set, since $F$ has no interior. 
 
Since $F$ is a Carleman set, there is a harmonic function $h$ in  $\B^n_R,$ such that 
$$
	(h-\varphi)(x)\to 0, \quad \mbox{as} \quad |x|\to R \quad \mbox{in} \quad F.
$$
In particular, for all $\theta'\in F^\prime,$  
$$
	\big(h-\varphi\big)(r, \theta')\to 0, \quad \mbox{as} 
		\quad r\nearrow R.
$$

\end{proof}


\section{PDE's}

In this section $\Omega$ is a domain in $\R^n, \, n>2$ and $L$ is an elliptic partial differential  operator of the form (\ref{L}), satisfying the conditions in the introduction on $\Omega.$ In the introduction, we defined $L$-harmonic and $L$-superharmonic functions.  If $L$ is defined on all of $\R^n$ and $u$  is an $L$-harmonic function on $\R^n,$ we say that $u$ is an {\it entire $L$-harmonic function}. According to our earlier definition, by an $L$-potential on a domain $\Omega\subset\R^n,$ we mean an $L$-superharmonic function $u \ge 0$ on $\Omega$ for which every $L$-harmonic minorant $\ge 0$ vanishes identically \cite{Br1958}. Notice that the constant function $0$ is an $L$-potential.

Denote by $\mathcal H_L(U)$ the real vector space of $L-$harmonic functions on  an open set $U\subset\Omega$
and by $\mathcal H_L$  the sheaf of solutions $u$ of $Lu=0$ on open subsets  of $\Omega,$ where we refrain from including $\Omega$ in the notation for simplicity. 

In Herv\'e \cite{Herve} on page 245 (see also \cite{Br1960}), it is asserted that 
$\mathcal H_L$ is a (Brelot) harmonic space.  We shall denote the corresponding potential theoretic notions by including the prefix $L.$ Thus, we shall speak of $L-$harmonic functions, $L-$thinness, etc.  When $L=\Delta,$ where $\Delta$ is the Laplace operator, we obtain classical potential theory.  Thus, $\Delta-$harmonic functions are classical harmonic functions,  $\Delta-$thinness is classical thinness, etc.

We  need the following theorem of Herv\'e \cite[Theorems 36.1 and 36.3]{Hervethesis}.

\begin{theorem}\label{thinness}
A subset $E$ of $\Omega$  is $L-$thin at a point of $\Omega$ if and only if it is $\Delta-$thin at that point and $E$ is $L$-polar if and only if it is $\Delta-$polar.
\end{theorem}
This states that thinness for harmonic functions is the same as thinness for $L-$harmonic functions. 
For this reason, for the remainder of this section, we shall use the terms thin and polar, without the prefex $L$ or $\Delta.$

We wish to show the existence of an $L-$potential $>0.$ 
Our assumption that $c\le 0$ in equation (\ref{L}) implies that the constant function $1$ is an $L-$superharmonic function $>0,$ which is $L-$harmonic if and only if $c=0.$
In fact, our assumption that there is an  $L-$superharmonic function $u\ge 0,$ which is not $L-$harmonic, implies the existence of an $L-$potential $>0.$  This follows immediately from the Riesz Representation Theorem \cite[Theorem 2.2.2]{CC}, which states that every $L-$superharmonic function $u \ge 0$ on a domain $U$ has a unique representation as the sum $u=p+h$ of an $L$-potential $p$ and an $L$-harmonic function $h.$ Since $u$ is not $L$-harmonic, $p\not=0.$ Since $p$ is a potential, $p$ is an $L-$superharmonic function $\ge 0.$
By Lemma \ref{min}, an $L-$superharmonic function $\ge 0$ on a domain $\Omega$ is either $>0$ or $\equiv 0.$  Thus $p>0$ and we have confirmed the existence of an $L-$potential $>0.$  
Thus, under the assumptions in the introduction,  $\mathcal H_L$ is not only a Brelot space, but in fact a strong Brelot space.

An additional important property of the Brelot space $(\Omega,\mathcal H_L)$ is that it satisfies Axiom D (see \cite[p. 567]{Hervethesis}). 

By an {\it adjoint} $L-$harmonic function $u$ on an open set $U,$ we mean a solution $u$ of the adjoint equation $L^*u=0.$ The following axiom is important for proving approximation results. 

{\it Axiom A* (quasi-analyticity).} Every adjoint harmonic function $u$ on a connected open set $U,$ which vanishes on a neighbourhood of some point of $U,$ must vanish identically on $U.$

According to \cite{Aro}, the Brelot space $(\Omega,\mathcal H_L)$ satisfies Axiom A*.  

The following convergence theorem can be found in \cite[p. 17]{Br1969}.

\begin{theorem}\label{uniform}
In a Brelot space, for a sequence $u_n$ of harmonic functions, locally and uniformly bounded in an open set $U,$ pointewise convergence implies that the limit is harmonic. 
\end{theorem}

By a closed subset of $\Omega$ we mean a subset which is closed in the topology of $\Omega.$
We say that a subset $X$ of $\Omega$ is $\Omega-$bounded if it is relatively compact in $\Omega,$ that is, if its closure in the topological space $\Omega$ is compact. If the context is clear, we shall omit the prefex $\Omega.$ 
If $F$ is a closed subset of $\Omega,$ by a hole of $F,$ we mean a bounded component of $\Omega\setminus F.$ 
We denote by  $\widehat F$  the union of $F$ with all of its holes. Clearly, $\widehat F$ has no holes and if $K$ is a compact subset of $\Omega$ then $\widehat K$ is also compact. 

The famous Mergelyan Theorem gives a complete characterization of  those compact subset $K$ of the complex plane $\C,$ such that each function continuous on $K$ and holomorphic on $K^\circ$ can be uniformly approximated by polynomials. Since entire functions can be represented by power series,  approximation by polynomials on compact sets  is equivalent to approximation by entire functions. Thus, Mergelyan's Theorem yields a characterization of  compact sets $K$ of the complex plane, for which every function $f\in C(K)$ satisfying the Cauchy-Riemann equation $\overline\partial f=0$ on $K^\circ,$ can be approximated by functions $g$ satisfying $\overline\partial g$ on all on $\C.$

For a closed subset $F\subset\Omega,$ we denote by $A_L(F)$ the family of continuous functions on $F$ which are $L$-harmonic on the interiour $F^\circ.$ 
It is natural to call  $F$ an {\it $L-$Mergelyan} (respectively {\it $L-$Carleman}) set in $\Omega$ if for every $u\in A_L(F)$,  and every $\epsilon,$ which is a  positive constant (respectively, positive continuous function on $F$)  there exists a function $h$, $L-$harmonic on $\Omega$, such that 
$$|h-u|<\epsilon\quad \textit{ on } F.$$
 The following theorem, which is a special case of  a result of Gardiner, Goldstein and GowriSankaran \cite[Theorem. 4]{GGG1994_compact}, characterizes the compact $L-$Mergelyan subsets of $\Omega$.

\begin{theorem}\label{continuouslyLharm}
Let $L$ be as above and $K$ a compact subset of $\Omega.$ Then, the following are equivalent. 

(a) 
$K$ is an $L-$Mergelyan set in $\Omega.$ 

(b) $\Omega\setminus \widehat K$ and $\Omega\setminus K^\circ$ are thin at the same points of $K.$ 
\end{theorem}

\begin{corollary}\label{compact Mergelyan}
The compact  $L-$Mergelyan sets are the same as the compact $\Delta-$Mergelyan sets. 
\end{corollary}

Results on approximation by global harmonic functions in the context of more general harmonic spaces can be found in \cite{BH1978} and \cite{Han}.

\begin{corollary}\label{Mergelyan polar}
For $L$ and $\Omega$ as above, if $K$ is an $L-$Mergelyan set in $\Omega$ and $Q$ is a compact polar set in $\Omega,$ then $K\cup Q$ is an $L-$Mergelyan  set in $\Omega.$ 
\end{corollary}

\begin{proof}
First of all, since $Q$ is polar,
\begin{equation}\label{hatcirc}
	\widehat{K\cup Q}  = \widehat K\cup Q, \quad \mbox{and} \quad
		(K\cup Q)^\circ=K^\circ.
\end{equation}

Suppose $\Omega\setminus (K\cup Q)^\circ$ is thin at a point $x\in K\cup Q.$ Since $(\Omega\cup Q)^\circ\subset\widehat{K\cup Q},$ it follows that $\Omega\setminus \widehat{K\cup Q}$ is also thin at $x.$ 

Suppose conversely that $\Omega\setminus \widehat{K\cup Q}$ is  thin at $x\in K\cup Q.$ We wish to show that $\Omega\setminus (K\cup Q)^\circ$ is thin at $x.$ 
If $x$ is in a hole of $K,$
 then, indeed, $\Omega\setminus (K\cup Q)^\circ=\Omega\setminus K^\circ$ is thin at $x.$  If $x$ is not in a hole of $K,$ since $x\in K\cup Q,$ we have $x\in K\cup(Q\setminus\widehat K).$ Suppose $x\in Q\setminus \widehat K.$ Since adding a polar set does not cancel thinness, 
\begin{equation}\label{hatK}
	\big[\Omega\setminus \widehat{K\cup Q}\big]\cup(Q\setminus\widehat K) 
		= \Omega\setminus \widehat K
\end{equation}
is thin at $x$ which is in $\Omega\setminus \widehat K,$ which is absurd because an open set cannot be thin at one of its points. Thus, $x\not\in Q\setminus \widehat K$ and so $x\in K.$ By (\ref{hatK}) we have $\Omega\setminus \widehat K$ is thin at $x\in K.$  Since $K$ is a $L-$Mergelyan set, this means that $\Omega\setminus K^\circ,$ which by (\ref{hatcirc}) is $\Omega\setminus (K\cup Q)^\circ,$ is thin at $x,$ which is what we wished to prove.  

\end{proof}

\medskip


It will  be convenient to have the following form of the Tietze extension theorem.

\begin{lemma}\label{Tietze c} If $X$ is a normal (topological) space, $0<c\le +\infty$ and $\phi:A\to(-c,+c)$ is a continuous function on a closed subset $A\subset X,$ then $\phi$ has a continuous extension $\Phi:X\to (-c,+c).$
\end{lemma}

\begin{proof}
The Tietze theorem is usually stated for $c=+\infty,$ but this form is equivalent, since $(-\infty,+\infty)$ and $(-c,+c)$ are homeomorphic. Now, this is a simple consequence of \cite[Theorem 8.2.11]{Singh}.

\end{proof}

By a {\it chaplet} $C$ in an open set $\Omega\subset\R^n,$ we mean the union of a family, locally finite in $\Omega,$ of  pairwise disjoint, connected, compact $L-$Mergelyan subsets of $\R^n,$ contained in $\Omega,$ whose complements in $\R^n$ are connected.

Note that $C$ is closed in $\Omega,$ since it is the union of a locally finite family of closed subsets of $\Omega.$ The following theorem furnishes a method for constructing interesting Carleman sets.

\begin{theorem} \label{beads} Let $\Omega$ be a domain in $\R^n, \,  n\ge 3,$ and $L$ an operator as above in $\Omega.$ Let $F$ be a closed polar subset of $\Omega$ and $C$ a chaplet in $\Omega.$ Then $E=F\cup C$ is an $L-$Carleman set in $\Omega.$ 
\end{theorem}

One way to prove the theorem would be to show that $L$ is the Laplace-Beltrami operator for a Riemannian metric on $\Omega.$ The  theorem would then simply be a special case of Theorem \ref{BG}. We, however, prefer to give a direct proof in the next subsection.


\subsection{Proof of Theorem \ref{beads}}

We need some construction and some lemmas before proving the theorem. 

\begin{lemma}\label{Kexhaustion}
Given a chaplet $C$ in a domain $\Omega\subset \R^n,$ there exists a regular exhaustion 
\[\mathcal K_1\subset  \mathcal K_2^\circ\subset \mathcal K_2\subset \mathcal K_3^\circ\subset\cdots,
\]
of $\Omega$ by compact $L-$Mergelyan
subsets of $\Omega,$ such that $C\cap\partial \mathcal K_j=\emptyset, \, j=1,2,\ldots.$
\end{lemma}

\begin{proof}
Let us call the compact sets in the definition of the chaplet $C$ the {\it parts} of $C.$ 
We may construct a regular exhaustion $(K_j)_{j=1}^\infty$ of $\Omega,$ whose members are smoothly bounded, such that, denoting by $C_j$ the union of those parts of $C$ which meet $\partial K_j,$ we have $C_j\subset K_{j+1}^\circ,$ and $C_{j+1}\cap K_j=\emptyset,$ for $j=1,2,\ldots.$ Now, for each $j=1,2,\ldots,$ let $V_j$ be a neighbourhood of $C_j,$ consisisting of the union of finitely many open balls of radius $r_j,$ centered at points of $C_j,$ where $r_j$ are chosen decreasing to zero so rapidly that $\overline V_j\cap (C\setminus C_j)=\emptyset$ and $\overline V_{j+1}\cap \partial K_j=\emptyset.$ Since each component of $V_j$ meets $\partial K_j,$ 
the sets $\mathcal K_j=K_j\setminus V_j$ have the required properties. 

\end{proof}

\begin{lemma}\label{OmegaMergelyan cup Mergelyan}
Let $K$ be a compact $L-$Mergelyan subset of a domain $\Omega\subset\R^n,$
 with $\widehat K=K$ relative to $\Omega$ and $Q\subset\Omega$ a compact $L-$Mergelyan subset of 
$\R^n,$ 
with $\widehat Q=Q$ relative to $\R^n$
and $K\cap Q=\emptyset.$ Then, $K\cup Q$ is an 
$L-$Mergelyan subset of $\Omega.$ 
\end{lemma}

\begin{proof}
Note that both $K$ and $Q$ are $L-$Mergelyan subsets of $\Omega,$  since the property of being a $L-$Mergelyan subset of a domain is hereditary in the sense that, if $\Omega_1\subset\Omega_2$ are domains in $\R^n$ and $Q\subset \Omega_1$ is a compact $L-$Mergelyan subset of $\Omega_2,$ then it is also a $L-$Mergelyan subset of $\Omega_1.$

To conclude the proof, it is sufficient, as in the proof of Lemma \ref{Mergelyan cup Mergelyan}, to show that  $(K\cup Q)^\wedge=K\cup Q$ in $\Omega.$ 
Thus, it is sufficient to show that each component of $\Omega\setminus (K\cup Q)$  is unbounded in $\Omega.$ 

Let $U$ be a component of $\Omega\setminus K$ and set $Q_U=U\cap Q.$ We claim the open subset that $U\setminus Q=U\setminus Q_U$ of $\Omega\setminus(K\cup Q)$ is connected. To see this, fix $x\in U\setminus Q_U.$ Since $U$ is (pathwise) connected, for an arbitrary $y\in U,$ we may choose  a path $\sigma_y$  in $U$ from $x$ to $y.$ If $\sigma_y\cap Q_U=\emptyset,$ set $\Sigma_y=\sigma_y.$ Otherwise, 
since $\widehat Q=Q$ in $\R^n$,
noting that $Q_U$ is a compact subset of $U,$  there exists $W,$ a neighbourhood of $Q_U$ in $U$ which is smoothly bounded in $\R^n,$ with $\overline W\subset U\setminus(\{x\}\cup\{y\}).$ The set $W$ has finitely many components and their closures are the components of $\overline W.$  Let $\overline W_1,\ldots, \overline W_m$ be those components  which intersect $\sigma_y.$
For each $j,$ let $X_j$ be the outer boundary of $\overline W_j,$ that is the boundary of the $\R^n-$unbounded component of $\R^n\setminus W_j.$ For each $j,$ $\sigma_x$ has a first point $p_j$ and a last point $q_j$ in $X_j.$ Let $\Sigma_y$ be the connected subset of $U\setminus Q_U$ obtained from $\sigma_y$ by replacing, for each $j,$ the portion $\sigma_j$ of $\sigma_y$ from $p_j$ to $q_j$ by the connected set $X_j.$   The set $\Sigma_y$ is a connected subset in $U\setminus Q_U.$ Since 
$$
	U\setminus Q_U=U\setminus K=\bigcup_{y\in U\setminus K}\Sigma_y
$$
is the union of connected sets having the point $x$ in common, the open subset $U\setminus Q_U$ of $\Omega\setminus (K\cup Q)$ is indeed connected as claimed. Since each component $U$ of $\Omega\setminus K$ is unbounded in $\Omega$ and $Q_U$ is bounded in $\Omega,$ each $U\setminus U_Q$ is unbounded in $\Omega.$ Now each component of $\Omega\setminus (K\cup Q)$ intersects (and thus contains) such an $U\setminus U_Q.$ Thus each component of $\Omega\setminus (K\cup Q)$ is unbounded in $\Omega.$ 

\end{proof}

We cannot strengthen the previous lemma by merely assuming that $K$ and $Q$ are disjoint Mergelyan sets in $\Omega.$ For example, if $\Omega=\R^n\setminus\{0\},$ $S_j=\{x\in\R^n:|x|=j\},$ $K=S_1\cup S_2$ and $Q=S_3,$ then  $K$ and $Q$ are disjoint compact $L-$Mergelyan subsets of $\Omega.$ However, $K\cup Q$ is not an $L-$Mergelyan subset of $\Omega.$

The following lemma is a particular case of the previous one, but we shall give a more direct proof.

\begin{lemma}\label{Mergelyan cup Mergelyan}
The union of two disjoint compact $L-$Mergelyan sets in $\R^n,$ having connected complements, is again an $L-$Mergelyan set in $\R^n$. 
\end{lemma}

\begin{proof}
Let $K_j, \, j=1,2$ be disjoint 
$L-$Mergelyan sets in $\R^n,$ with $\widehat K_j=K_j, \, j=1,2.$ Since thinness is a local condition, it is sufficient to show that $(K_1\cup K_2)^\wedge=K_1\cup K_2.$ Thus, it is sufficient to show that $X=\R^n\setminus (K_1\cup K_2)$ is connected.  One easily shows this by verifying that, for a fixed $x\in X$, the set of $y\in X,$  which can be connected to $x$ by a path in $X,$ is both open and closed in $X.$

\end{proof}

\begin{lemma} \label{Mergelyan2}
Let  $\Omega$ be a domain in $\R^n, n\ge 3,$ $F$ a polar set in $\Omega$ and $L, C,  \mathcal K_j, j=1, 2, \ldots$ as above. For $j=1,2,\ldots,$ the sets
$(\mathcal K_j\cup F\cup C)\cap \mathcal K_{j+1}$ are $L-$Mergelyan sets in $\Omega.$
\end{lemma}

\begin{proof} Since the family of compacta whose union is $C$ is locally finite in $\Omega,$ 
by Lemma \ref{Mergelyan cup Mergelyan}, $C\cap (\mathcal K_{j+1}\setminus \mathcal K_j)$ 
is an $L-$Mergelyan set in $\R^n.$ Then, by Lemma \ref{OmegaMergelyan cup Mergelyan} $(\mathcal K_j\cup C)\cap \mathcal K_{j+1}$ is an $L-$Mergelyan set in $\Omega.$ Finally, by Corollary \ref{Mergelyan polar} 
$(\mathcal K_j\cup F\cup C)\cap \mathcal K_{j+1}$ is an $L-$Mergelyan set in $\Omega.$ 

\end{proof}

\begin{proof}{(proof of  Theorem \ref{beads})}
Note that $E$ is  closed in $\Omega$ as it is the union of two closed subsets of $\Omega.$ 
We may rename the compact sets $(\mathcal K_k)_{k=1}^\infty$ in the exhaustion obtained by Lemma \ref{Kexhaustion} (for the chaplet $C$) such that $\mathcal K_1=\emptyset.$

Let $(\epsilon_k)_{k=0}^\infty$ 
be a sequence of positive numbers, such that
$$
\epsilon_{k+1}+\epsilon_{k+2}+\cdots<\epsilon_k, \quad \mbox{for} \quad k=0,1,2, \ldots,
$$
and
$$
2\, \epsilon_{k-1}\le\inf\{\epsilon(y)\, : \,y\in 
(F\cup C)\cap (\mathcal K_{k+1}\setminus \mathcal K_k^\circ)\}, \quad \mbox{for} \quad k=1,2,\ldots. 
$$
Also, suppose $u\in A(E).$

Set $w_0=u$ on 
$(F\cup C)\cap \mathcal K_2.$
By Corollary \ref{Mergelyan polar},
there is a $v_1\in\mathcal H_L(\Omega),$ such that 
$$
|v_1-w_0|<\epsilon_2,\quad \textit{ on } \quad (C\cup F)\cap \mathcal K_2.
$$
Note that on 
$ (F\cup C)\cap \mathcal K_2,$ one has 
$$|v_1-u|=|v_1-w_0|<\epsilon_2<\epsilon_0.$$ 
We define a function $w_1$ in 
$A\big(\mathcal K_2 \cup (C\cap (\mathcal K_3\setminus \mathcal K_{2}))\cup (F\cap \partial \mathcal K_{3})\big),$ 
by setting $w_1=v_1$ on $\mathcal K_2,$ and $w_1=u$ on 
$(C\cap (\mathcal K_3\setminus \mathcal K_{2}))\cup (F\cap \partial \mathcal K_{3}).$ 
By  the Tietze Lemma \ref{Tietze c}, extend $w_1$ 
to $\widetilde w_1,$
from 
$
\mathcal K_2 \cup (C\cap (\mathcal K_3\setminus \mathcal K_{2}))\cup (F\cap \partial \mathcal K_{3})
$ 
continuously to 
$
(\mathcal K_2 \cup F\cup C)\cap \mathcal K_{3}
$ 
such that 
\[
|\widetilde w_1-u|<\epsilon_2
\textit{ on } 
(F\cup C)\cap (\mathcal K_3\setminus \mathcal K_2^\circ).
\]
 We, however, keep denoting the extension $\widetilde w_1$ by $w_1$ to simplify the notation.
Noting that $w_1\in A\big(\mathcal K_2 \cup (C\cap (\mathcal K_3\setminus \mathcal K_{2}))\cup (F\cap \partial \mathcal K_{3})\big),$ by Lemma \ref{Mergelyan2}, there is a function $v_2\in\mathcal H_L(\Omega),$ such that 
$$
|v_2-w_1|<\epsilon_3, \quad \textit{ on } \quad  
(\mathcal K_2 \cup F\cup C)\cap \mathcal K_{3}.
$$
Thus, $|v_2-v_1|<\epsilon_3$ on 
$\mathcal K_2$ and $|v_2-u|<\epsilon_3+\epsilon_2<\epsilon_1$ on 
$
(F\cup C)\cap (\mathcal K_3\setminus \mathcal K_2^\circ).
$

We proceed thus, by induction, to construct a sequence $v_k\in \mathcal H_L(\Omega),$ such that  
$$|v_k-v_{k-1}|<\epsilon_{k+1} \textit{ on } \mathcal K_{k},$$ 
and 
$$
|v_k-u|<\epsilon_{k-1} \textit{ on }
(F\cup C)\cap (\mathcal K_{k+1}\setminus \mathcal K_k^\circ),
$$
for $k=1, 2, \ldots.$

The sequence 
$(v_k)_{k=1}^\infty$ 
is uniformly Cauchy on compact subsets of $\Omega,$ and hence converges uniformly on compact subsets of $\Omega$ to a function $v=\lim v_k$ in $\mathcal H_L(\Omega).$ Theorem \ref{uniform}, yields $v\in\mathcal H_L(\Omega).$

Now suppose $y\in F\cup C.$ By exhaustion, there exists a fixed $k=1, 2, \ldots$ such that $y\in 
(F\cup C)\cap (\mathcal K_{k+1}\setminus \mathcal K_k^\circ).
$ Therefore,
$$
	|v(y)-u(y)| = 
		|v_k(y)-u(y)+\sum_{j=k+1}^\infty (v_{j}-v_{j-1})(y)| \le 
	|v_k(y)-u(y)|+\sum_{j=k+1}^\infty |v_{j}(y)-v_{j-1}(y)|.
$$
Note that since 
$y\in \mathcal K_{k+1},$
 it is in 
 $\mathcal K_j$ 
 for every $j\ge k+1,$ hence
$|v_{j}(y)-v_{j-1}(y)|<\epsilon_{j+1}.$ Also, 
$$|v_k(y)-u(y)|<\epsilon_{k-1}.$$
Thus, the above yields
$$
	|v(y)-u(y)|	\le 
	\epsilon_{k-1}+\sum_{j=k+1}^\infty \epsilon_{j+1}\le 
	\epsilon_{k-1}+\epsilon_{ k+1}\le 2\epsilon_{k-1}\le\epsilon (y).
$$
\end{proof}


\section{Some special domains for $L$-harmonic functions}

\medskip

\subsection{Vertical line domains}
Let us say that a domain $U\subset\R^n=\R^{n-1}\times\R$ is a vertical line domain with 
base $ U',$ if 
$U\cap (\{x'\}\times\R)$ is non-empty and connected, for each $x'=(x_1, \cdots, x_{n-1})\in U'$ and is empty for each $x'\not\in U'.$ 
For $x'\in U^\prime,$ set 
\begin{equation}\label{cddefinition}
\begin{array}{c}
	c(x')=\inf\left\{y:(x', y)=(x_1,\ldots, x_{n-1},y)\in U\right\}\\
	d(x')=\sup\left\{y:(x', y)=(x_1,\ldots, x_{n-1},y)\in U\right\}
\end{array}.
\end{equation}\label{Ucd}
The vertical line domain $U$ can be described as 
\begin{equation}
	U=\{x=(x', y)\in U^\prime\times\R: c(x')<y<d(x')\}.
\end{equation}
Note that $U^\prime$ is the projection $\pi(U)$ of $U$ 
into $\R^{n-1}.$

\begin{lemma} If $U$ is a vertical line domain, the functions $c$ and $d$ given by (\ref{cddefinition}) are respectively  upper and lower semicontinuous. Moreover, we have $c:U'\to[-\infty, +\infty)$ and $d:U'\to(-\infty,+\infty].$
\end{lemma}

\begin{proof}
First of all, we claim that $c(x')<+\infty,$ for each $x'\in U'.$ Indeed,  (\ref{cddefinition}) implies there is a finite value $y,$ such that $c(x')<y$ and consequently $c(x')<+\infty.$ 

Now suppose, to obtain a contradiction, that there is a point $x_0'\in U^\prime,$ at which $c$  is not upper semicontinuous.  Then, there is a neighbourhood $V_0'\subset U'$ of $x_0'$ and a finite value $y_0,$ with 
\begin{equation}\label{semicontinuity}
	  c(x_0')<y_0<c(x'),	\quad	\mbox{for all}	\quad	x'\in V_0'\setminus\{x_0'\}.
\end{equation}
From the definition of $c(x_0'),$ there is a finite value $y_1,$ with  $c(x_0')<y_1<y_0,$ and $(x_0',y_1)\in U.$
Since $U$ is open, it follows that there is a neighbourhood $V'_1\subset V_0'$ of $x_0',$ such that $(x',y_1)\in U,$ for all $x'\in V_1'$ and therefore $c(x')<y_1<y_0,$ for all $x'\in V_1'.$ 
This contradicts (\ref{semicontinuity}) and completes the proof of upper semicontinuity. 

The proof that the function  $d(x')$ is lower semicontinuous is similar. 

\end{proof}

\begin{lemma}\label{smooth}
Let $V$ be a domain in $\R^n,$ and let $c:V\to [-\infty,+\infty)$ and  $d:V\to(-\infty,+\infty]$ be respectively upper and lower semicontinuous functions with $c<d.$ 
Let $X_0=\emptyset.$  
Choose a point $x_1\in V$ and set  $X_1=\{x_1\}.$  Let $X_1, X_2, \ldots,$ be a regular exhauston of $V$ by compact sets. 
Then, there are sequences of smooth functions $c_k: V\to(-\infty,+\infty)$ and $d_k: V\to(-\infty,+\infty),$ 
on $V$ such that $c_k\to c$, $d_k\to d$ and 
$$
	c(x)<c_{k-1}(x)<c_k(x)<d_{k-1}(x)<d_{k}(x)< d(x),\quad \mbox{ for all } x\in X_{k-1},	\quad	\mbox{ for all } k=1,2,\ldots	
$$
\end{lemma}

\begin{proof} 

By Theorem 3 in \cite{T} (see also \cite{HK}) the lower semicontinuous function $d$ is the limit of a monotonically increasing sequence of continuous extended-real functions $f_j:V\to (-\infty,+\infty].$ We  claim  that we may assume that the $f_j$ are strictly less than $d.$ 
Since the problem is purely topological, we may temporarily replace $(-\infty,+\infty]$ by $(0,1].$ Thus, we have continuous functions $f_j:V \to (0,1],$ and a lower semicontinuous function $d:V\to (0, 1],$ with $f_j$ monotonically converging to $d.$ Replace the functions $f_j$ by the functions $jf_j/(j+1),$ which are strictly less than $d.$ This proves our claim on  $(0, 1]$ and hence on $(-\infty,+	\infty].$  

Similarly we can find a sequence $(g_j)$ of real-valued continuous functions on $V$ that are strictly greater than $c$ and $g_j\rightarrow c.$

We shall construct, by induction, a subsequence $u_k= f_{j_k}, \, k=1,2,\ldots,$ of the sequence $(f_j)$ and $v_k= g_{j_k}, \, k=1,2, \ldots,$ of the sequence $(g_j)$ such that
\begin{equation}\label{strict}
	c(x)<v_{k}(x)<v_{k-1}(x)<u_{k-1}(x)<u_{k}(x)<d(x), 	\quad	\forall \,	x\in X_{k-1},	\quad 	\forall \,	k=1,2,3, \ldots, 
\end{equation}
Since $c(x_1)<d(x_1),$ there are  $j_1<j_2$ such that, setting $v_1=g_{j_1}, v_2=g_{j_2, }, u_1=f_{j_1}, u_2=f_{j_2}$ we have  
$$c(x_1)<v_2(x_1)<v_1(x_1)<u_1(x_1)<u_{2}(x_1)<d(x_1).$$
 We have (\ref{strict}) for $k=2.$

Suppose, we have $u_k=f_{j_k}, v_k=g_{j_k}, \, k=2,\ldots, \ell-1$ satisfiying (\ref{strict}).
For each $x\in V,$ we may choose $j_{x}>j_{\ell-1},$ such that 
$$c(x)<g_{j_x}(x)<g_{j_{\ell-1}}(x)<f_{j_{\ell-1}}(x)<f_{j_x}(x)<d(x).$$ 
There is a $\delta=\delta_x>0,$ such that the ball $B(x; \delta)$ is contained in $V$ and 
$$c(y)<g_{j_x}(y)<g_{j_{\ell-1}}(y)<f_{j_{\ell-1}}(y)<f_{j_x}(y)<d(y) \quad   \mbox{ for all }\quad y\in B(x; \delta).$$ 
By compactness, there are finitely many points $x_1,\ldots,x_m\in X_{\ell-1},$
finitely many integers $\nu_1,\ldots,\nu_m>j_{\ell-1}$ and finitely many radii $\delta_1,\ldots,\delta_m,$ such that, setting $j_\ell=\max\{\nu_1,\ldots,\nu_m\},$ we have $X_{\ell-1}\subset B(x_1; \delta_1)\cup\cdots\cup B(x_m; \delta_m)$ and 
$$
		c(x)<v_\ell(x)<v_{\ell-1}(x)<u_{\ell-1}(x)<u_{\ell}(x)<d(x), 	\quad	\mbox{for all}	\quad	x\in X_{\ell-1},
$$
This completes the inductive construction of the sequences $(u_k)$ and  $(v_k).$

Now for $k=1, 2,3 \cdots,$ let 
$$
2\epsilon_{k-1}=\min\{\min_{X_{k-1}}(u_k-u_{k-1}), \min_{X_{k-1}}(u_{k-1}-v_{k-1}), \min_{X_{k-1}}(v_{k-1}-v_{k})\}.
$$ 
Let $d_{k-1}, c_{k-1}$ be polynomials such that $|d_{k-1}-u_{k-1}|<\epsilon_{k-1}$ and $|c_{k-1}-v_{k-1}|<\epsilon_{k-1}$  on $X_{k-1}$ (which exists by a simple application of the Stone-Weierstrass theorem). Clearly, 
$$
c(x)<c_{k-1}(x)<c_k(x)<d_{k-1}(x)<d_{k}(x)< d(x),\quad \mbox{ for all } x\in X_{k-1}.
$$
\end{proof}

\medskip

For a vertical line domain, $U,$
and a function $h:U\to \R,$
let us define the lower, respectively upper, vertical cluster sets for $x^\prime\in \R^{n-1}$
: 
$$
	C_c(h,x^\prime) = \left\{
w\in[-\infty,+\infty ]: \textit{ there exists a sequence } y_j
\searrow c(x^\prime), \,\textit{ such that } h(x^\prime,y_j)\to w
\right\};
$$
and  
$$
	C^d(h,x^\prime) =
	 \left\{
w\in[-\infty,+\infty ]: \textit{ there exists a sequence } y_j
\nearrow d(x'), \,\textit{ such that } h(x^\prime,y_j)\to w
\right\}.
$$

\begin{theorem}\label{vertical EG}
Let $U\subset \R^n, n\ge 3,$ be a  vertical line domain and let $L$ be an operator of the form (\ref{L}) on $U.$ Moreover, 
let $E'=\cup_j E'_j$ be a countable union  of subsets of $U'$ such that the
Euclidean closure of each $E'_j$ has
empty fine interior in $U'.$
Then, for every function $\varphi$ continuous on $U,$ there is an $L-$harmonic  function $u$ on $U$ such that,
\begin{equation}\label{zero}
\big(u-\varphi\big)(x^\prime,y)\to 0, \quad \mbox{if} 
\quad  y\searrow c(x') \quad \mbox{or}
		\quad y\nearrow d(x'), \quad \mbox{for all} \quad  x^\prime \in E'
\end{equation}
and
\begin{equation}\label{cluster}
	C_c(u, x^\prime) =
	C^d(u,x^\prime) = [-\infty,+\infty],
\quad \mbox{for a.e.}  \quad  x^\prime \in U'\setminus E'.
\end{equation}
\end{theorem}

\begin{proof} 
To prove this, we begin by constructing an exhaustion.
Let $X_t$ be a regular exhaustion of $U'$ by smoothly bounded compact sets, with $X_t\subset X^\circ_{t+1}, \, t=1, 2, \ldots.$ Replacing the sequences in Lemma \ref{smooth} by appropriate subsequences, we may assume that $c_k(x')<d_k(x')$ for all $x'\in X_k$ and  $k=1,2,\ldots.$ Set $Q_0=\emptyset$ and define
$$
Q_k=\{(x',y): x'\in X_ k, \, c_k(x')\le y\le d_k(x') \},
$$ 
for $k=1, 2, \ldots.$ 

Consider the closed subsets $F_k, \, k=1, 2,\dots$ of $U$ defined by
$$
F_k= \{(x', y)\ :\ x'\in E_k^\prime\cap X_k, \, c(x')<y \le c_k(x')\}
\cup
 \{(x', y)\ :\ x'\in E_k^\prime\cap X_k, \, d_k(x') \le y<d(x')\}.
$$
Note that $c<c_k$ and  $d_k<d,$ as we have indicated.
Clearly, $F=\cup F_k$ is a closed subset of $U$ as it is a locally finite (countable) family of closed sets in $U$.

Now, let $A'_1\subset A'_2\subset\cdots,$ be an increasing sequence of closed 
subsets of $U'\setminus E',$ such that $m(U'\setminus A'_{k})<1/{k},$ where $m$ is Lebesgue measure on $\R^{n-1}.$ We define closed sets $C_k$ 
as follows
$$
	C_{k}= \left\{\big(x', c_{k}(x')\big): x'\in A'_{k}\cap X_k\right\} \cup \left\{\big(x', d_{k}(x')\big): x'\in A'_k\cap X_k\right\} ,
$$
and set 
 $C=\cup_{k=1}^\infty C_{k}.$

We shall show that $Z=F\cup C$ is an $L-$Carleman set in $U.$ Set 
$$
	  Z_k= Z\cap (Q_k\setminus Q_{k-1}^\circ),   \quad K_k=Q_{k-1}\cup Z_k,. 
$$
Note that $\widehat K_k=K_k.$

There is no relation between the sequences $(E_k')$ and $(X_k).$ However, the corresponding sequences $(F_k\cup C_k)$ and $(Q_k)$ are related as follows.  Denoting the slab $S_k=\{(x',y):c_k(x')\le y \le d_k(x')\},$ we have that $(F_k\cup C_k)\subset U\setminus S_k^\circ$ and $Q_k\subset S_k.$

We shall construct a sequence of functions $P_k\in \mathcal H_L(U),$ which is uniformly Cauchy on the compact subsets $Q_k$ of $U.$ 
In order to carry out this plan, we wish to   use Theorem \ref{continuouslyLharm} to show that  $K_k$   is an $L-$Mergelyan
 set.  Thus, we need to show that $U\setminus\widehat K_k$ and $U\setminus K_k^\circ$ are thin at the same points of $K_k.$  

Clearly, if  $U\setminus K_k^\circ$ is thin at a point, then $U\setminus\widehat K_k$ is also thin at this point.
Conversely, suppose $U\setminus\widehat K_k$  is thin at a point $(x',y)$ of $K_k.$ Since $\widehat K_k=K_k,$ we have $U\setminus K_k$  is thin at  $(x',y).$ If $(x',y)\in K_k^\circ,$ then of course $U\setminus K_k^\circ$ is thin at $(x',y).$ There remains the case that $U\setminus K_k$ is thin at 
$$
	(x',y)\in \partial K_k =\partial Q_{k-1}\cup Z_k
$$
and it is sufficient to show that this case is impossible; that is, that $U\setminus K_k$ is thin at no point of $\partial K_k.$ Since thinness is local, it is sufficient to show that $\R^n\setminus K_k$ is thin at no point of $\partial K_k.$ 
Consider first the case that $(x',y)\in \partial Q_{k-1},$ which consists of three parts: 

$\bullet$ the  {\it side}  of $ Q_{k-1};$
that is,
$$\{(x',y)\in\partial Q_{k-1}:x'\in \partial X_k \},$$
and 

$\bullet$ the 
{\it bottom} and {\it top} of $Q_{k-1};$
that is,
$$
	\{(x',y):x'\in X_k^\circ, \, y=c_k(x') \} \cup \{(x',y):x'\in X_k^\circ, \, y=d_k(x') \}.
$$

A subset $K\subset \R^n$  is said to satisfy the {\it exterior cone condition} at a point $y\in\partial K$ if $y$ is the vertex of an open cone in $\R^n\setminus K.$ 
From the classical Poincar\'e-Zaremba  cone condition  (see e.g. \cite{GilbargTrudinger}), it is sufficient to show that   $\R^n\setminus K_k,$ contains an open cone at each point $(x',y)\in \partial Q_{k-1}.$ 
Since $X_k\times\R$ is smoothly bordered in $\R^n,$ it follows that 
$\R^n\setminus (X_k^\circ\times\R)$ has an open cone at each point $(x',y)\in\partial X_k\times\R,$ from which it follows that $\R^n\setminus K_k$ has an open cone at each point in the side of $Q_{k-1}.$  Similarly, it follows from the smoothness of $c_k$ and $d_k$ that $\R^n\setminus K_k$ has an open cone at each point in the bottom or top of $Q_{k-1}.$ 

Secondly, consider the case that $(x', y)\in Z_k,$ That is, we want 
 to show that $U\setminus K_k$ is thin at no point of $Z_k.$ Suppose, to obtain a contradiction, that  $U\setminus K_k$ is thin at a point $(x',y)$ of $Z_k.$ Since thinness is local,  $\R^n\setminus (K_k\setminus Q_k)=\R^n\setminus Z_k$ is thin at $(x',y),$ so $(x',y)$ is in the fine interior of $Z_k,$ since the $E_j'$ and $C_j$ are locally finite families and the $F_j$ are increasing, it follows that, for some $\ell,$ 
$$
	{Z_k} \subset Q_{k}\cap (F_\ell\cup C_1\cup \cdots \cup C_\ell),
$$
so $(x',y)$ is in the fine interior of $F_\ell\cup C_1\cup\cdots\cup C_\ell.$ Since these are disjoint closed sets, the fine interior of the union is the union of the fine interiors. Since $E_\ell'$ has no fine interior, it follows from Theorem \ref{Fuglede manifolds} that $F_\ell$ has no fine interior.  Hence, some $C_j$ has non-empty fine interior.  
But this  contradicts Lemma \ref{fine sphere}, which concludes the proof that $U\setminus\widehat K_k$ and $U\setminus K_k$ are thin at the same points of $K_k.$ By Theorem \ref{continuouslyLharm}, $K_k$ is an $L-$Mergelyan set. 

A close examination of our argument reveals that, for $k=1,2,\ldots$ and $\ell\ge k,$ also the sets 
$$
	 Z_k\cup\cdots \cup Z_\ell, \quad	\quad	\mbox{and}	\quad	Q_{k-1}\cup
(Z_k\cup\cdots \cup Z_\ell),	
$$
are $L-$Mergelyan sets.

Now, we prove our claim that  $Z$ is a $L-$Carleman set. The idea of the proof is taken from the proof of Carleman's Theorem in 
\cite{Gai} and is originally due to Brelot.

 Let $f\in A_L(Z)=C(Z)$ and $\epsilon$ a positive continuous function on $Z.$ Let $e_0=0$ and $e_n, \, n=1,2,\ldots,$ be a decreasing Cauchy sequence of positive numbers, such that 
$$
2(e_{n+1}+e_n)<
\min
\{\epsilon(x):\, x\in Z_{n+1}
\}, \quad
 \textit{and }	\quad	 \sum_{k=n+1}^\infty e_k<e_n. 
$$
By Theorem \ref{continuouslyLharm} there exists $P_0\in \mathcal H_L(U)$ such that
$$
|P_0-f|<e_2\quad \textit{ on } \, Z\cap Q_3=Z_1\cup Z_2\cup Z_3.
$$
We shall construct by induction $P_1, P_2,\ldots \in  H_L(U),$ which satisfy the condition: 
\begin{equation}\label{induction3}
\begin{array}{cccl}
|P_n(x)-P_{n-1}(x)|& < &e_n, 	&	\forall x\in Q_n;\\

|P_n(x)-f(x)|& <&	e_n + 2e_{n-1},	&	\forall x\in Z\cap (Q_{n+1}\setminus Q_n^\circ)=Z_{n+1};\\

|P_n(x)-f(x)|& < &e_{n+1},	&	\forall x\in 	Z\cap(Q_{n+2}\setminus Q_{n+1}^\circ)=Z_{n+2}
.
\end{array}
\end{equation}
Set $P_1=P_0.$ Then $P_1$ satisfies (\ref{induction3}). Next, suppose we have $P_n$ which satisfy (\ref{induction3}), for $n=1,\ldots,\ell-1.$ 
Put $\varphi=0$ on $Z\cap \partial Q_\ell,$  $\varphi=1$ on $Z_{l+2}=Z\cap (Q_{\ell+2}\setminus Q_{\ell+1}^\circ).$
By  Lemma \ref{Tietze c} (the Tietze Extension Theorem), we may
extend $\varphi$ continuously 
to $Z\cap (Q_{\ell+2}\setminus Q_{l}^{\circ})$
such that
it remains positive and bounded by $1.$  
$$
h(x) = 
\left\{
\begin{array}{ll}
	P_{\ell-1}(x),	&	\forall x\in Q_\ell;\\
	P_{\ell-1}(x)+\varphi(x)(f(x)-P_{\ell-1}(x)),	&	\forall x\in {Z}\cap (Q_{\ell+2}\setminus Q_\ell^{\circ}).
\end{array}
\right . 
$$

By Theorem \ref{continuouslyLharm} 
(since  $Q_\ell\cup(Z\cap (Q_{\ell+2}\setminus Q_\ell^\circ))$ is an $L-$Mergelyan set in $U$), there is $P_\ell\in \mathcal H_L(U)$ such that
$$
	|P_\ell(x)-h(x)|	<	e_{\ell+1}, 	\quad 	\forall x\in Q_\ell\cup \left(Z\cap(Q_{\ell+2}\setminus Q_\ell)\right).
$$
We claim that $P_\ell$ satisfies (\ref{induction3}).  Indeed,
\begin{align*}
	|P_\ell(x)-P_ {\ell-1}(x)|=|P_\ell(x)-h(x)|<& e_{\ell+1}< e_\ell, 	\qquad \forall x\in Q_\ell;\\
	|P_\ell(x)-f(x)| = |P_\ell(x)-h(x)| <& e_{\ell+1}, \qquad \forall x\in Z\cap (Q_{\ell+2}\setminus Q_{\ell+1}^\circ);\\
	|P_{\ell}(x)-f(x)|	\le |P_\ell(x)-h(x)|+|h(x)-f(x)|\le & e_{\ell+1}+2|P_{\ell-1}(x)-f(x)|\\
\leq&  e_{\ell+1} + 2e_\ell\leq e_\ell+2e_{\ell-1}, \qquad \forall x\in Z\cap (Q_{\ell+1}\setminus Q_\ell).
\end{align*}
Since $P_\ell$ satisfies (\ref{induction3}), this concludes the inductive proof of the existence of the sequence $P_1, P_2, \ldots .$ 
Moreover, for every $j$ and all $\ell>k>j,$ 
$$
|P_\ell(x)-P_k(x)|	\le \sum_{n=k+1}^\ell \left|P_n(x)-P_{n-1}(x)\right|	<	 \sum_{n=k+1}^\ell e_n, \quad	\forall x\in Q_j,    
$$
and since $e_1,e_2,\ldots$ is a Cauchy sequence, it follows that the sequence $(P_n)$ is uniformly Cauchy on $Q_j,$ for every $j.$ 
Thus, $P_n\to g,$ for some $g\in \mathcal H_L(U).$ 

To show that $|g-f|\le 1$ on $Z,$ fix $x\in Z$ and let $n$ be the unique integer such that $x\in Z_{n+1}.$  
$$
	|g(x)-f(x)|	=	\left|P_n(x)	+	\sum_{k>n}\left(P_k(x)-P_{k-1}(x)\right)-f(x)\right|	\le 	
		\sum_{k>n}e_k	+	(e_n + 2e_{n-1})	<	2(e_{n-1}+e_n)	<	e(x). 
$$
This concludes the proof that $Z$ is an $L-$Carleman set in $U.$

\bigskip
To prove the claim about the vertical behaviour, let 
$(\eta_k
)_{k=1}^\infty$ be a dense sequence in $(-\infty,+\infty).$  We define a continuous function $\psi$ on $Z=F\cup C,$ by setting $\psi=\varphi$ on $F$ and 
$\psi=\eta_k$ on $C_k$ for $k=1,2,\ldots.$ 
For $x=(x', y)\in U,$ we  define the function $\epsilon:U\to (0,+\infty),$ by setting 
$$
\epsilon(x) = 
\epsilon\left((x', y)\right) =  
\min\left\{
\chi\left(
(x', y), (x', c(x')\right), 
\chi \left((x', y), (x', d(x')\right)
\right\},
$$
Of course $\epsilon(x)$ tends to zero, as $y\searrow c(x'),$ or $y\nearrow d(x'),$ for all $x'\in U'.$

Finally, by the definition of an $L-$Carleman set, there is an $L-$harmonic function $u$ on $U,$ such that
$$|\psi(x)-{ u}(x)|<\epsilon(x),\quad \textit{ for all } \quad x\in Z.$$ 
 
Since $\psi=u$ on $F$ and $\psi=\eta_k$ on $C_k,$ $u$ satisfies  the required properties (\ref{zero}) and (\ref{cluster}).

\end{proof}


\subsection{Radial domains} 
As mentioned earlier, 
Euclidian space $\R^n$  is a Riemannian manifold with the usual distance $d_n$ as Riemannian metric and we consider the sphere $S^{n-1}$ as a Riemannian manifold with the restriction of $d_n$ as Riemannian metric.  This induces a harmonic structure on $S^{n-1}$ and we shall sometimes add the prefix ``spherically" to the corresponding potential theoretic notions on $S^{n-1},$ for example spherically thin, spherically polar and spherically fine topology. Moreover, when it is clear that we are speaking of spherical potential theory on $S^{n-1},$ we may omit the prefex ``spherical''.

By  Theorem \ref{Fug}, if we consider $\pi:\R^n\setminus \{0\}\to S^{n-1}, \, x\mapsto x/|x|$ (i.e. the radial projection), then $E'$ in $S^{n-1}$ is spherically thin at a point $y\in E',$  if and only if $E=\pi^{-1}(E')$ is thin at each point of $\pi^{-1}(y)$ in $\R^n\setminus \{0\}.$

Let $U'$ be a domain in $S^{n-1}$ and $r:U'\to [0,+\infty), \, R:U'\to (0, +\infty]$ be continuous functions, with $r<R.$ 
We call the domain 
$$
U=(U^\prime,r,R)=
\left\{(\theta,\rho)\in S^{n-1}\times \R^+:\theta\in U^\prime, \, r(\theta)<\rho<R(\theta)
\right\}
$$
a {\it radial line domain}.

\begin{theorem}\label{starshell}
Let $U\subset\R^n, n\ge 3,$ be a radial line domain and let $L$ be an operator of the form (\ref{L}) on $U.$
Moreover, 
let $F^\prime=\cup_j F'_j$ be a countable union  of subsets of $U^\prime$ such that the
Euclidean closure of each $F'_j$ has
empty (spherically) fine interior in $U^\prime.$

Then, for every function $\varphi$ continuous on $U,$ there is an $L-$harmonic function $h$ on $U$ such that,
$$
	\big(h-\varphi\big)(\theta,\rho)\to 0, \quad \mbox{as} \quad \rho
	\searrow 
	r(\theta), \quad 
	\mbox{or} \quad \rho
	\nearrow
	 R(\theta) \quad ;\quad \textit{ for all }\quad \theta \in F^\prime.
$$
\end{theorem} 

\begin{proof}
This is not an immediate consequence of Theorem \ref{bundle}, because $U$ is not a trivial harmonic bundle, since the fibre over a point $\theta\in S^{n-1}$ is not the whole line $(-\infty,+\infty),$ but rather an open interval $\big(r(\theta),R(\theta)\big)$ in $[0,+\infty].$ However, the only property  of the projection $\pi$ which we needed in the proof of Theorem \ref{bundle} was that the pull back of 
sets with empty fine interior has empty fine interior (by
Corollary \ref{fine interior})
which is the case here. 

Let $(r_k)_{k=1}^\infty$ be a decreasing and $(R_l)_{l=1}^\infty$ be an increasing sequence of continuous functions on $U'$ such that $r_k\searrow  r$ and $R_k\nearrow R.$ We set 
$$
	F_k = \big\{(\theta,\rho): \theta\in F_k',\, \,
	\rho\in \big(r(\theta),R(\theta)\big)\setminus \big(r_k(\theta),R_k(\theta)\big)\big\},
$$
and $F=\cup_{k=1}^\infty F_k.$ 
Then, as in the proof of Theorem \ref{bundle}, the set $F$ is a Carleman set for harmonic approximation in $U.$ The rest of the proof is like that of Theorem \ref{bundle} but much simpler. 

\end{proof}


\subsection{Strictly Starlike Domains}

We shall say that a domain $D\subset\R^n$ is {\it strictly starlike} with respect to the origin $\bf 0$ if for every $p$ in $\overline D$, and every $0 \le \lambda < 1$,  $\lambda p$ is in the interiour of $D$. 
We only consider starlike domains with respect to the origin so we need not repeat ``with respect to the origin" again. 
An example of a domain, which is starlike but not {\it strictly} starlike is the slit plane $\C\setminus [0, +\infty).$

We first show that a domain $U$ is strictly starlike if and only if there is a unique continuous function $R:S^{n-1}\to [0,+\infty],$ such that, in polar coordinates,
$$
U=
\{x=(\theta, \rho): 0\le \rho<R(\theta), \, \theta\in S^{n-1}\}.
$$

The ``if part" is clear as the above $U$ is a strictly starlike domain. For the the ``only if part" we let $R(\theta)=sup\{\rho : (\theta, \rho) \in U\},$ for $\theta\in S^{n-1}.$ The sup could well be $+\infty,$ for some values of $\theta,$ as the domain is not necessarily bounded.

\begin{lemma}\label{continuous}
$R:S^{n-1}\to[0,+\infty]$ is continuous, where $[0,+\infty]$ has the usual order topology. 
\end{lemma}

\begin{proof}
We first show that $R$ is lower semicontinuous in $S^{n-1}$. Suppose, to obtain a contradiction, that there is a point $\theta\in S^{n-1},$ a finite value $s< R(\theta)$ and a sequence $\theta_j\to \theta$ (in $S^{n-1}$) such that $R(\theta_j)<s.$ Since $U$ is open, it follows that no point in the segment $\{(\theta, r): s<r<R(\theta)\}$ lies in $U,$ which contradicts the definition of $R(\theta).$ Note that the argument is valid even if $R(\theta)=+\infty.$

We now need to show the upper semicontinuity of $R$ in $S^{n-1}$. Fix $\theta\in S^{n-1}.$ If $R(\theta)=+\infty,$ upper semicontinuity at $\theta$ is automatic. Suppose $R(\theta)<+\infty$ 
and suppose, to obtain a contradiction, that there is a sequence $\theta_j\to \theta$ and $\epsilon>0,$ such that $R(\theta_j)>c       R(\theta)+\epsilon.$ Then, for $s=R(\theta)+\epsilon,$ $(\theta, s_j)\in U.$ 
However, $(\theta, s)\not\in U.$ Therefore, $(\theta, s)\in \partial U.$ This contradicts the definition of $R(\theta),$ as $U$ is a strictly starlike domain.

\end{proof}
Taking into account such a parametrization, we shall denote a strictly starlike domain $U$ as $U=(S^{n-1}, R).$ We wish to study the radial behaviour of harmonic functions on strictly starlike domains and we begin with a theorem on the radial limits of continuous functions. 
By a measurable set in a topological space, we mean, unless otherwise specified, a Borel measurable set; by a measurable function between two topological spaces, we mean a Borel measurable function and by a Borel measure on a  locally compact Hausdorff space, we mean  a measure defined on the class of Borel sets, such that the measure of compact sets is finite.

\begin{theorem}
Let $U=(S^{n-1}, R)$ be a strictly starlike domain in $\R^n, n\ge 2,$ and $\Psi:U\to (-\infty, +\infty)$ a 
continuous function.
For $\theta\in S^{n-1}$ we denote $\Psi^*(\theta)\in[-\infty, +\infty]$ the radial limit of $h$ on the ray $\{(\theta, r) : 0\leq r <R(\theta)\}.$
Let $A\subset S^{n-1}$ be the set of all $\theta\in S^{n-1}$ for which $\Psi^*(\theta)$ exists.
Then $A$ is a measurable set and the function $\Psi^*:A\to [-\infty, +\infty]$ is of Baire class $0$ or $1.$
\end{theorem}

\begin{proof}
If the radial limit along some ray is infinite then it is either $-\infty$ or $+\infty.$ 
Indeed, suppose $\Psi(\theta, r)\to\infty$ as $r\to R(\theta),$ and there are sequences $r^+_j\to R(\theta)$ and $r^-_j\to R(\theta),$ such that $\Psi(\theta, r^+_j)\to+\infty$ and $\Psi(\theta, r^-_j)\to-\infty.$ By the intermediate value theorem, there is a sequence $r^\circ_j\to R(\theta),$ such that $\Psi(\theta, r^\circ_j)=0.$ This contradicts the assumption that $\Psi(\theta, r)\to\infty$ as $r\to R(\theta).$

Consider the homeomorphism
$$
	h:U\longrightarrow\B^n, \quad x=(\theta,\rho)\longmapsto \big(\theta,\chi(0,\rho)\big),
$$
where $\chi$ denotes the {\it chordal distance}.
Let $(r_j)_{j=1}^\infty$ be a sequence increasing to $1$ in $[0,1).$ We introduce the continuous functions 
$$
	R_j: S^{n-1}\to [0,+\infty), \quad \theta\mapsto |h^{-1}(\theta,r_j)|, \quad j=1, 2, \ldots.
$$
Note that $0<R_j\le R_{j+1},$ and $R_j\to R,$ as $j\to \infty.$ 
 Then, $\theta\in A$ if and only if
$$
	\sup\{\chi\big(\Psi(\theta,R_j(\theta)),\Psi(\theta,\rho)\big):R_j(\theta)\le \rho<R(\theta)\}\to 0, \quad \mbox{as} \quad j\to \infty. 
$$
Putting $\Q_j=\Q\cap [R_j(\theta),R(\theta)),$ this is equivalent to 
$$
	\sup
	\left\{
	\chi\big(\Psi(\theta,R_j(\theta)),\Psi(\theta,\rho)\big):\rho\in \Q_j
	\right\}
	\to 0, \quad \mbox{as} \quad j\to \infty.
$$
Thus,
$$
	A = \bigcap_{i=1}^\infty\bigcup_{j=1}^\infty\bigcap_{\rho\in\Q_j}\left\{\theta\in S^{n-1}:\chi\big(\, \Psi(\theta, R_j(\theta)), \Psi(\theta, \rho)\,\big)\le \frac{1}{i}\right\},
$$
and since each function $\chi\big(\, \Psi(\theta, R_j(\theta)), \Psi(\theta, \rho)\big)$ is continuous in $\theta,$ it follows that $A$ is measurable and in fact an $F_{\delta\sigma\delta}-$set.

Finally, the function $\Psi^*:A\to[-\infty,+\infty]$ is of Baire class $0$ or $1,$ since it is the limit of a sequence of continuous functions; that is,
$$
	\lim_{j\to\infty}\Psi\big(\theta,R_j(\theta)\big)=\Psi^*(\theta), \quad \mbox{for all} \quad \theta\in A.
$$

\end{proof}

\begin{lemma}\label{exhaustion}
Let $U=(S^{n-1}, R)$ be a strictly starlike domain in $\R^n, n\ge 2.$ Then, there exists  a regular exhaustion of $U$ by compact sets $Q_1, Q_2, \ldots,$
which are starlike and satisfy the exterior cone condition. 
\end{lemma}

\begin{proof} 
First of all, there exists a strictly increasing sequence of continuous functions $\phi_k:S^{n-1}\to(0,+\infty), \, \phi_k(\theta)\nearrow R(\theta).$ Fix $k,$  $\theta_0\in S^{n-1},$  $\epsilon,$ with $\epsilon<\phi_{k+1}(\theta_0)-\phi_k(\theta_0)$ and $0<\alpha<1.$
Set $\rho_0=\phi_{k+1}(\theta_0), \, \rho_\epsilon=\rho_0-\epsilon, \, x_0=(\theta_0,\rho_0)$ and $x_\epsilon=(\theta_0,\rho_\epsilon).$ 
Denote by $C(\theta_0,\alpha,\epsilon),$ the open cone with vertex $x_\epsilon:$
$$
	C(\theta_0,\alpha,\epsilon)  =	\{x=(\theta,\rho)\ :\  \rho>\rho_\epsilon,\ |x-x_\epsilon|<\alpha
	(\rho-\rho_\epsilon)
	\}.	
$$ 
Choose $\epsilon$ so small that the closed ball of center $x_0$ and radius $\epsilon$ is disjoint from the compact set
 $K_k=\{(\theta, 	\rho):\rho\le\phi_k(\theta)\}.$ Now, choose $\alpha$ so small that the cone $C(\theta_0,\alpha,\epsilon)$ is disjoint from $K_k.$ The cone $C(\theta_0,\alpha,\epsilon)$ is an open neighbourhood of $x_0,$ 
and note that $x_0$ is an arbitrary point on the boundary of $K_{k+1}.$
  By compactness, there is a finite number of such cones $C_{k,1},	\ldots, C_{k,m_k}$ which covers $\partial K_{k+1}.$ 
Set  
$$
	Q_k=\R^n\setminus (C_{k,1}\cup\cdots\cup C_{k,m_k}).
$$

\end{proof}

For a strictly starlike domain $U=(S^{n-1}, R)$ in $\R^n,$  
a function $u:U\to[-\infty,+\infty]$ and a point $\theta\in S^{n-1},$ denote by $C(u,\theta)$ the {\it radial cluster set}
$$
C(u,\theta) = \left\{
w\in[-\infty,+\infty ]: \textit{ there exists a sequence } r_j(\theta)\nearrow R(\theta), \,\textit{ such that } u(r_j,\theta)\to w
\right\}.
$$

\begin{theorem}\label{starlike EG}
Suppose $U=(S^{n-1}, R)$ is a strictly starlike domain in $\R^n, \, n\ge 3,$ and let $L$ be an operator of the form (\ref{L}) on $U.$ Moreover, 
let $E^\prime=\cup_j E'_j$ be a countable union  of subsets of $S^{n-1}$ such that the
Euclidean closure of each $E'_j$ has
empty fine interior in $S^{n-1}.$
Then, for every function $\varphi$ continuous on $U,$ there is an $L-$harmonic  function $h$ on $U$ such that, 
$$ 
		\lim_{\rho \nearrow R(\theta)}\big(h-\varphi\big)((\theta, \rho))=0, \quad \mbox{for all} \quad \theta\in E^\prime,
$$
and
$$
	C(h,\theta) = [-\infty,+\infty], \quad \mbox{for a.e.}  \quad  \theta \in S^{n-1}\setminus E^\prime.
$$
\end{theorem}

\begin{proof}

It follows from Lemma \ref{A union B} that 
we may assume that the $E^\prime_k$ are closed and increasing, noting that the Euclidean closure is the same as the $S^{n-1}-$closure.
We may assume that the sequence $(E^\prime_j)_{j=1}^\infty$ is infinite, because if $E^\prime=E_1^\prime\cup\cdots \cup E_k^\prime,$ we may set $E_j^\prime=E_k^\prime,$ for all $j\ge k.$ 
Let $Q_k, \, k=1,2,\ldots,$ be an exhaustion of $U,$ as in Lemma \ref{exhaustion}. Since each $Q_k$ is starlike, there is a strictly increalsing sequence of defining functions $R_k:S^{n-1}\to (0,+\infty), \, R_k\nearrow R,$ such that 
$$
	Q_k = \{(\theta,\rho):\theta\in S^{n-1}, \, 0\le\rho\le R_k(\theta)\}.
$$   

Consider the closed subsets of $U$  defined by
$$
F_k= \{(\theta,\rho):\theta\in E_k^\prime, \, R_k(\theta)\le\rho<R(\theta)\}, \quad \, k=1, 2,\dots.
$$
and set $F=F_1\cup F_2 \cup \ldots .$

Now, let $A'_1\subset A'_2\subset\cdots,$ be an increasing sequence of compact subsets of $S^{n-1}\setminus E',$ such that $m(S^{n-1}\setminus A'_j)<1/j,$ where $m$ is Hausdorff $(n-1)-$measure on $S^{n-1}.$ We define closed {\it spherical caps} $C_j$ as follows
$$
	C_j= \left\{\big(\theta, R_j(\theta)\big): \theta\in A'_j\right\},
$$
Note that $C_j\subset\partial Q_j.$ Set $C=\cup_{j=1}^\infty C_j.$ 

A modification of the proof of Theorem \ref{vertical EG} shows that $F\cup C$ is an
$L-$Carleman set in $U.$
To prove the claim about the radial behaviour, let 
$(c_j)_{j=1}^\infty$ be a dense sequence in $(-\infty,+\infty).$  We define a continuous function $\psi$ on $F\cup C,$ by setting $\psi=\varphi$ on $F$ and $\psi=c_j$ on $C_j$ for $j=1,2,\ldots.$ 

For $x=(\theta,\rho)\in U,$ define the function $\epsilon:U\to (0,+\infty),$ by setting 
$$
\epsilon(x) = \epsilon\left((\theta,\rho)\right) = \chi\left((\theta,\rho),(\theta,R(\theta)\right),
$$
 where $R:S^{n-1}\to(0,+\infty]$ is the defining function for $U,$ which is continuous by Lemma \ref{continuous}. Then, $\chi$ is also
continuous, since $(0,+\infty)\times(0,+\infty)\subset \overline\C\times\overline\C$ and the spherical distance $\chi$ is continuous on the Riemann sphere $\overline\C.$ Of course $\epsilon(x)=\epsilon(\theta,\rho)$ tends to zero, as $\rho\to R(\theta),$ for all $\theta\in S^{n-1}.$

Finally, by the definition of a $L-$Carleman set, there is an $L-$harmonic function $h$ on $U,$ such that
$$|\psi(x)-h(x)|<\epsilon(x),\quad \textit{ for all } x\in  F\cup C.$$ 

\end{proof}


\section*{Acknowledgment}
The junior author would like to thank the senior author for introducing the problem and for his generous support. Also he is grateful to Dr. Dmitry Jakobson and Dr. Javad Mashreghi for the financial support, and hospitality, during his postdoctoral studies at McGill University.

\end{document}